\documentclass[11pt,letterpaper]{amsart}
\usepackage{amsmath,amssymb,amscd,amsfonts,mathrsfs}
\usepackage[alphabetic]{amsrefs}
\usepackage{tabularx}
\usepackage{latexsym}
\usepackage{color}
\usepackage{tikz}
\usepackage[all,cmtip,line]{xy}
\usepackage{enumitem}
\usepackage[only,llbracket,rrbracket]{stmaryrd}
\usepackage{relsize}
\usepackage[bbgreekl]{mathbbol}
\DeclareSymbolFontAlphabet{\mathbb}{AMSb}
\DeclareSymbolFontAlphabet{\mathbbl}{bbold}
\usepackage[normalem]{ulem}

\DeclareSymbolFont{euletters}{U}{eur}{m}{n}
\DeclareSymbolFont{eufrakletters}{U}{euf}{m}{n}
\DeclareFontFamily{U}{wncy}{}
    \DeclareFontShape{U}{wncy}{m}{n}{<->wncyr10}{}
    \DeclareSymbolFont{mcy}{U}{wncy}{m}{n}
    \DeclareMathSymbol{\Sha}{\mathord}{mcy}{"58}

\newcommand{\Arrow}[1]{%
\parbox{#1}{\tikz{\draw[->](0,0)--(#1,0);}}
}
\newcommand{\shortrightarrow}{\Arrow{.15cm}}

\newcommand{\ilim}{\underleftarrow{\text {\rm lim}}}

\newcommand{\Alb}{\text{\rm Alb}}
\newcommand{\Pic}{\text{\rm Pic}}

\newcommand{\Res}{\text{\rm Res}}

\newcommand{\Sym}{\text{\rm Sym}}

\renewcommand{\dim}{\text {\rm dim}}

\newcommand{\Lie}{\text {\rm Lie}}
\newcommand{\Rlim}{\text {\rm Rlim}}
\newcommand{\Spec}{{\operatorname{Spec\,}}}

\newcommand{\Prism}{{\mathlarger{\mathbbl{\Delta}}}}

\newcommand{\Proj}{\text{\rm Proj}}

\newcommand{\an}{{\text{\rm an}}}

\newcommand{\rk}{\text{\rm rk\,}}
\renewcommand{\ker}{\text{\rm ker}\,}

\newcommand{\ad}{{\text{\rm ad}}}
\newcommand{\ed}{{\text{\rm ed}}}

\renewcommand{\Im}{\text {\rm Im }}

\renewcommand{\log}{\text{\rm log}}
\renewcommand{\th}{^{\text{th}}}

\newcommand{\Gal}{\text{\rm Gal}}
\newcommand{\hotimes} {\Hat\otimes}

\newcommand{\Sh}{\text{\rm Sh}}

\newcommand{\A}{\mathcal A}

\renewcommand{\AA}{\mathbb A}

\newcommand{\CC}{\mathbb C}

\DeclareMathSymbol{\fk}\mathord{eufrakletters}{"6B}

\newcommand{\F}{\mathbb F}

\newcommand{\GG}{\mathbb G}

\DeclareMathSymbol{\ei}\mathord{euletters}{"69}
\DeclareMathSymbol{\etau}\mathord{euletters}{"1C}
\DeclareMathSymbol{\eiota}\mathord{euletters}{"13}

\DeclareMathSymbol{\eK}\mathord{euletters}{"4B}

\renewcommand{\L}{\mathcal L}

\newcommand{\M}{\mathcal M}

\renewcommand{\O}{\mathcal O}

\newcommand{\Q}{\mathbb Q}

\newcommand{\RR}{\mathbb R}

\renewcommand{\SS}{\mathbb S}

\newcommand{\Z}{\mathbb Z}

\newcommand{\et}{\text{\rm\'et}}

\newcommand{ \iso} {\overset \sim \longrightarrow}

\newcommand{\Aut}{\text {\rm Aut}}

\newcommand{\lps}{[\![}
\newcommand{\rps}{]\!]}

\newtheorem{ithm}{Theorem}

\numberwithin{equation}{subsection}

\newtheorem{thm}[equation]{Theorem}

\newtheorem{cor}[equation]{Corollary}
\newtheorem{lemma}[equation]{Lemma}
\newtheorem{lem}[equation]{Lemma}
\newtheorem{prop}[equation]{Proposition}

 \theoremstyle{definition}
 \theoremstyle{definition}
 \theoremstyle{remark}
\newtheorem{rem}[equation]{\bf Remark}

\newtheorem{para}[equation]{\bf}

\theoremstyle{definition}

\begin{document}

\title{Essential dimension via prismatic cohomology }

\author{Benson Farb, Mark Kisin and Jesse Wolfson}
\address{Department of Mathematics, University of Chicago}
\email{farb@math.uchicago.edu}
\address{\vskip -.5 cm Department of Mathematics, Harvard }
\email{kisin@math.harvard.edu}
\address{ \vskip -.5 cm Department of Mathematics, University of California-Irvine}
\email{wolfson@uci.edu}


\thanks{The authors are partially supported by NSF grants DMS-1811772 (BF), DMS-1902158 (MK) and
DMS-1811846 and DMS-1944862 (JW)}



\begin{abstract}  For $X$ a smooth, proper complex variety we show that for $p\gg 0$, the restriction of the mod $p$ cohomology $H^i(X,\F_p)$ to any Zariski open has dimension at least $h^{0,i}_X$.  The proof uses the prismatic cohomology of Bhatt-Scholze.

We use this result to obtain lower bounds on the $p$-essential dimension of covers of complex varieties. 
For example, we prove the $p$-incompressibility of the mod $p$ homology cover of an abelian variety, confirming a conjecture of Brosnan for sufficiently large $p.$ By combining these techniques with the theory of toroidal 
compactifications of Shimura varieties, we show that for any Hermitian symmetric domain $X,$  there exist $p$-congruence covers that are $p$-incompressible.
\end{abstract}



\maketitle
\tableofcontents

\section{Introduction} Let $f:Y \rightarrow X$ be a finite map of complex algebraic varieties. The {\em essential dimension}
$\ed(Y/X)$ of $f$ is the smallest integer $e$ such that, over some dense open of $X,$ the map $f$ arises as the pullback of a map of varieties 
of dimension $e.$ The motivation for studying this invariant goes back to classical questions about reducing the number of 
parameters in a solution to a general $n\th$ degree polynomial. It first appeared in work of Kronecker \cite{Kronecker} and Klein \cite{Klein} on the quintic, and was formally defined by Tschebotarow \cite{Tschebotarow}, and in a modern context by Buhler-Reichstein \cite{BR}, \cite{BR2}. 

An idea which goes back to Arnol'd \cite{Arnold}, \cite{Arnold2}, is to use characteristic classes to obtain a lower bound for essential dimension: if $f$ has covering group $G$ then the cohomology of $G$ sometimes contributes to the cohomology of $X$ - such classes on $X$ are called {\em characteristic}. If they contribute to the cohomology of $X$ in some degree $i,$ and $X$ is affine, 
then $f$ cannot arise as a pullback of a $G$-cover of dimension $< i.$ 
The main difficulty with this approach is that to get a lower bound for $\ed(Y/X)$ one needs to know that (some of) the classes coming from $G$ continue to be nontrivial on any Zariski open in $X.$ Indeed in {\it loc.~cit.~}Arnol'd was able to address the question 
only without restricting to open subsets. 
For the universal $S_n$-cover, this problem was solved by Buhler-Reichstein \cite{BR2}. Following a suggestion of 
Serre, they showed the Stiefel-Whitney classes which arise in that situation are nonzero at the generic point. 

In this paper we introduce a new method, which solves this restriction problem in many cases. In particular, it allows us to give lower bounds on $\ed(Y/X)$ in many cases when $X$ is {\em proper}. 
Previously, lower bounds on the essential dimension of coverings 
of proper varieties were known only in very special cases \cite{CT,FKW,FS}. In fact our results apply to the so called 
{\em $p$-essential dimension} 
$\ed(Y/X;p)$ ($p$ a prime), where one is allowed to pull back the covering not just to Zariski opens, but to auxiliary coverings of degree prime to $p$  \cite{ReiYo}. 
A first example of our results is the following, which appears as Cor.~\ref{cor:ellabessdimlowerbounds} below. 

\begin{ithm}\label{ithm:elemab} Let $X$ be a smooth proper complex variety of maximal Albanese dimension, 
and $Y \rightarrow X$ its mod $p$ homology cover. 
Then for $p \gg 0,$ $Y\rightarrow X$ is {p-incompressible}, i.e. 
$$ \ed(Y/X;p) = \dim X. $$
\end{ithm}
\noindent 
Recall that the mod $p$ homology cover is the covering corresponding to the maximal elementary abelian $p$-group quotient of the fundamental group of $X,$ and that $X$ has maximal Albanese dimension if the image of its Albanese map has dimension equal to $\dim X.$ 
The condition $p \gg 0$ can be made effective in terms of the behavior of $X$ and its Albanese map under reduction mod $p.$

There are clearly a plethora of varieties to which Theorem~\ref{ithm:elemab} applies. This includes the following: 

\begin{itemize}
\item  $X$ an abelian variety.  
\label{ex:abvar}
\item $X=C_1\times \cdots \times C_r$ for curves $C_i$ with ${\rm genus}(C_i)\geq 1.$ 
\item Certain cocompact, $n$-dimensional ball quotients for $n\geq 2$; see \cite[p.~167, Cor.~5.9]{BW}.
\label{ex:modp}
\end{itemize}

When $X$ is an abelian variety, this confirms (almost all of) a conjecture of Brosnan \cite[Conj.~6.1]{FS}, 
which was previously known only for either a very general abelian variety, or in dimension at most 3 for a positive density set of primes \cite{FS}, 
compare also \cite[Cor.~p14]{Bogomolov}. The result for a product of curves was known only for a generic product of elliptic curves, by a result of Gabber \cite{CT}. If $X$ is a generic abelian variety, or a product of generic curves, then the effective version of Theorem~\ref{ithm:elemab} 
actually implies that the conclusion of the theorem holds for all $p.$ We leave it to the reader to find further examples. 

Our results are not actually limited to elementary abelian $p$-covers. One can apply them to $G$-covers for any finite group $G,$ but 
the conditions which have to hold are rather more restrictive:

\begin{ithm}\label{ithm:Gcover} Let $X$ be a smooth, proper complex variety, $G$ a finite group, and $Y \rightarrow X$ a $G$-cover. 
Suppose that $X$ has unramified good reduction at $p,$ and let $i < p-2.$ If 
$H^0(X, \Omega^i_X) \neq  0$ and the map 
$H^i(G,\F_p) \rightarrow H^i(X, \F_p)$
is surjective then
$$ \ed(Y/X;p) \geq i.$$
\end{ithm}

We refer the reader to \ref{para:goodredn} for the definition of unramified good reduction. 
This is a condition that holds for $p$ sufficiently large. If $X$ is defined over a number field $F,$ the condition means that there is an 
unramified prime of $F$ over $p$ at which $X$ has good reduction. The condition holds for {\em all} primes if $X$ is a generic member 
of a sufficiently good moduli space, see \ref{para:genericgoodredn} below.

We also work with open varieties, and not just proper ones, in which case the formulation of the above results involves logarithmic differentials, 
see \ref{prop:essdimlowerbounds}, \ref{thm:ellabessdimlowerbounds} below.
Note that the condition on surjectivity in Theorem \ref{ithm:Gcover} is quite restrictive. If $X$ is an \'etale $K(\pi,1),$ then 
there always exists a $G,$ and a $G$-cover satisfying the condition, but this is not true in general.

In \S\ref{subsec:toric} we explain how to apply our results to torus torsors over abelian schemes. Here the fundamental group is a {\em generalized Heisenberg group} - a central extension of finitely-generated free abelian groups - and there is a natural notion of ``reduction mod $p$'' for such a group. We use Theorem \ref{ithm:Gcover} to show that for $p \gg 0,$ 
the corresponding covers are $p$-incompressible; see \ref{prop:edTtorsor}. The verification of the surjectivity hypothesis in Theorem \ref{ithm:Gcover} involves a somewhat intricate calculation of the cohomology of mod $p$ Heisenberg groups. 
This uses, in particular, that these groups carry a kind of mod $p$ weight filtration, which is somewhat suggestive of the mod $p$ weight filtration introduced by Gillet-Soul\'e \cite{GS}.

We then use Theorem \ref{ithm:Gcover} to deduce the $p$-incompressibility of certain covers of locally symmetric varieties. These have the form $\Gamma_1\backslash X \rightarrow \Gamma\backslash X,$ where $X$ is a Hermitian symmetric domain and $\Gamma$ is an arithmetic lattice in the corresponding real semisimple Lie group. Thus, $\Gamma$ is a congruence subgroup of $G(\Q)$ for  a semi-simple $\Q$-group $G.$ We consider {\em principal $p$-congruence covers} which 
means, roughly speaking, that there are no congruences mod $p$ involved in the definition of $\Gamma,$ and $\Gamma_1$ is the subgroup of elements which are trivial mod $p.$ A sample of our results is the following. 

\begin{ithm}\label{ithm:locsym} Suppose $X$ is an irreducible Hermitian domain and that the $\Q$-rank 
of $G$ (the rank of its maximal $\Q$-split torus) is equal to its $\RR$-rank. 
For any principal $p$-congruence cover $\Gamma_1\backslash X \rightarrow \Gamma \backslash X$, we have 
$$ \ed(\Gamma_1\backslash X \rightarrow \Gamma \backslash X; p ) = \dim X,$$ 
provided $p$ satisfies the following conditions if $X$ is not a tube domain:
\begin{itemize} 
\item If $X$ is of classical type, then $p > \frac{3}{2}\dim X.$ 
\item If $X$ is of type $E_6$, then $p$ is sufficiently large.
\end{itemize}
\end{ithm}

Note that for {\em any} irreducible Hermitian domain there are many examples where Theorem~\ref{ithm:locsym} applies. Our results actually apply to many examples of quotients of {\it reducible} Hermitian domains by irreducible lattices,  for example the case of Hilbert modular varieties. See Lemma \ref{lem:rationalcocharacter} and the discussion in \ref{para:rationalcocharacter}.
The study of the ($p$-)incompressibility of congruence covers goes back to work of Klein \cite{KleinLetter,KleinAbel}, and our results here add to a recent body of work \cite{FKW,FS,FKW2,BF}, and prove new cases of \cite[Conjecture 1]{BF}  in the context of locally symmetric varieties.

Theorem \ref{ithm:locsym} is {\em not} proved directly using our results on characteristic classes. Rather, we 
relate principal $p$-congruence covers to the covers of torus torsors over abelian varieties, with covering group a 
generalized mod $p$ Heisenberg group mentioned above, and we then apply 
our results about the latter covers. The connection between these two kinds of covers makes use of the theory of toroidal compactifications of Shimura varieties. 

We remark that when $X$ {\em is} a tube domain, the corresponding torus torsor is just a torus; the abelian variety is $0$-dimensional. 
In this case, one does not need our results on characteristic classes (this is why there is no restriction on $p$ in this case). 
The theorem can be deduced from a result of Burda \cite{Burda} on coverings of tori. 
The result when $X$ is a tube domain has also been independently proven by  Brosnan-Fakhruddin \cite{BF}, who used the fixed-point method instead of Burda's results. 

The conditions on $G$ and $p$ in Theorem \ref{ithm:locsym} are completely different from those considered in \cite{FKW,FS}. For example, in many cases when the results of {\it loc.~cit.~}apply, they give $p$-incompressibility not for sufficiently large primes, but only for an explicit set of primes of positive density. They are also restricted to groups of classical type. On the other hand, proper varieties pose no special problem for these results, whereas they cannot be handled by Theorem \ref{ithm:locsym}. This also shows, that one should not expect the bounds on $p$ in Theorem \ref{ithm:locsym} to be sharp, as there are also many cases which are covered by {\it both}  Theorem \ref{ithm:locsym} and \cite{FKW,FS}, but where the latter results impose no similar lower bound on $p.$

As mentioned above, our lower bounds on essential dimension rely on results 
that assert that cohomology classes do 
not vanish on restriction to any Zariski open. An example of such a statement is the following, which appears as Cor.~\ref{cor:goodredn} below:

\begin{ithm}\label{ithm:nonzerorestrict} Let $X$ be a smooth, proper, complex variety, with unramified good reduction at $p.$ 
Let $i < p-2,$ and $W \subset X$ a Zariski open. Then the image of the  map 
$$ H^i(X, \F_p) \rightarrow H^i(W, \F_p)$$ 
has dimension at least $h^{0,i}_X = \dim H^0(X, \Omega_X^i).$
\end{ithm}

To get a feel for what such a statement entails, suppose that the integral cohomology $H^\bullet(X,\Z)$ is torsion free. 
Then Theorem \ref{ithm:nonzerorestrict} asserts that certain classes 
in the image of $H^i(X, \Z) \rightarrow H^i(W, \Z)$ are not divisible by $p$.   Although this is a statement about the topology of complex algebraic varieties, it appears to be out of reach of classical methods. The analogue with $\Q$-coefficients can be proved using Hodge theory, but this does not suffice for applications to essential dimension. We remark that the theorem could also be formulated in terms of the  {\it stable cohomology} introduced by Bogomolov \cite[p.2]{Bogomolov}, which is the image of the cohomology of $X$ in the cohomology of its generic point. 

For $X$ ordinary,  Theorem \ref{ithm:nonzerorestrict} was established by Bloch-Esnault \cite[Theorem 1.2]{BlochEsnault} using more classical $p$-adic Hodge theory. Our proof of Theorem~\ref{ithm:nonzerorestrict} makes use of {\em prismatic cohomology}, recently introduced by Bhatt-Scholze \cite{BhattScholze}. 
This is a cohomology theory that takes as an input a smooth formal scheme $\mathcal X$ over a $p$-adic base and, in some sense, 
interpolates between the mod $p$ (or more generally $p$-adic) \'etale cohomology of its generic fiber, and the de Rham cohomology of 
its special fiber $\mathcal X_k.$ Using it, we translate the statement of Theorem  \ref{ithm:nonzerorestrict} into 
the analogous statement for de Rham cohomology of $\mathcal X_k,$ and then into a statement about differentials using the Cartier isomorphism. In light of \cite[Theorem 1.2]{BlochEsnault}, this proof  shows the strength of the prismatic theory.

The theorem of Bloch-Esnault \cite[Theorem 1.2]{BlochEsnault} has played a role in a number of applications, including the study of torsion and divisibility in Chow groups as in \cite{Schoen,Totaro,Diaz},  Lefschetz type theorems \cite{PR} and  Galois actions on fundamental groups of curves \cite{HM}. We expect that Theorem~\ref{ithm:nonzerorestrict} should allow for strengthenings of many of these results.  For example, Scavia \cite{Sca} has recently applied Theorem~\ref{ithm:nonzerorestrict} to extend the main result of \cite{Schoen} from 
$p\equiv 1\text{\rm mod\,} 3$ to all primes $p>5.$ 
 
 To deduce Theorem \ref{ithm:Gcover} from Theorem \ref{ithm:nonzerorestrict} one considers the composite 
$$ H^i(G, \F_p) \rightarrow H^i(X, \F_p) \rightarrow H^i(W, \F_p).$$
Theorem \ref{ithm:nonzerorestrict} and the assumptions of Theorem \ref{ithm:Gcover} guarantee this map is nonzero. 
However if $Y|_W \rightarrow W$ 
arises from a covering of varieties of dimension $< i,$ then we may assume that these varieties are affine, and it follows 
that the above map must vanish since the cohomological dimension of affine varieties is at most their dimension. 

We also prove a variant of  Theorem \ref{ithm:nonzerorestrict} where we consider the image of the map 
$\wedge^i H^1(X,\F_p) \rightarrow H^i(W,\F_p)$ given by the cup product and restriction. This leads to 
Theorem \ref{ithm:elemab}. Using the trace map, it is easy to deduce from Theorem \ref{ithm:nonzerorestrict} a version where one allows restriction to a prime-to-$p$ covering of $W,$ so we obtain lower bounds on $p$-essential dimension.

The paper is organized as follows. In \S 2.1 we review the results we need from prismatic and crystalline cohomology. 
In \S 2.2 we prove Theorem \ref{ithm:nonzerorestrict} and its variants, and in \S 2.3 we apply this to characteristic classes to deduce lower bounds on $p$-essential dimension. Throughout \S 2 we work with schemes or formal schemes equipped with a normal 
crossings divisor, and we consider coverings of the complement. This greater level of generality is needed for the applications in \S 3.
In \S 3.1, we carry out an analysis of the mod $p$ cohomology of (generalized) finite Heisenberg groups. This is used in \S 3.2 to apply our 
results on characteristic classes to torus embeddings over abelian varieties, by verifying the 
surjectivity assumption in Theorem \ref{ithm:Gcover} (or more precisely its logarithmic analogue) in this case.
Most of \S 3.2 is actually concerned with a variant of Theorem \ref{ithm:nonzerorestrict} for torus embeddings over an abelian scheme, where we restrict not just to Zariski opens, but to analytic neighborhoods of the boundary. This is then applied in \S 3.3 to obtain our results on the $p$-essential dimension of congruence covers: the analytic neighborhoods map to our locally symmetric varieties and we can restrict the congruence covers to them.

\medskip
\noindent
{\bf Acknowledgments.} We thank Dave Benson, Bhargav Bhatt, H\'el\`ene Esnault, Keerthi Madapusi Pera, Alexander Petrov, Mihnea Popa, Gopal Prasad and Chris Rogers for useful discussions and suggestions.  We thank Patrick Brosnan, Najmuddin Fakhruddin, Federico Scavia and Peter Scholze for helpful comments on a draft.  We thank the anonymous referees for numerous helpful comments and suggestions.

\section{Essential dimension and characteristic classes}\label{sec:prelim}
\numberwithin{equation}{subsection}
\subsection{Review of mod $p$ cohomology}
\begin{para} In this subsection we review the results we will need on de Rham and prismatic cohomology. 
We begin with the former, see Deligne-Illusie [DI, 4.2.3]. 

Let $k$ be a perfect field of characteristic $p,$ and let $X$ be a smooth $k$-scheme. Suppose that $X$ is equipped with a normal crossings 
divisor $D \subset X.$   Let $\Omega^{\bullet}_{X/k}(\log D)$ denote the logarithmic de Rham complex with poles in $D.$ 

Let $X^{(1)} = X \times_{\Spec k, F_k} \Spec k$ and let $D^{(1)} = D \times_{\Spec k, F_k} \Spec k,$ where $F_k$ denotes the Frobenius 
on $k.$  Let $F_{X/k}:X \rightarrow X^{(1)}$ denote the relative Frobenius;  it is a finite flat map of $k$-schemes taking $D$ to $D^{(1)}.$ 
\end{para}

\begin{lemma}\label{lem:degdeR} Suppose that $(X,D)$ admits a lift\footnote{In fact for what follows only a lifting to the Witt vectors of length $2,$ $W_2,$ is required, but we will not need this greater generality.} 
to a smooth formal scheme $\tilde X$ over $W(k)$, equipped with a normal crossing divisor $\tilde D$ relative to $W(k).$ 
Then for $j < p,$ $H^0(X^{(1)}, \Omega^j_{X^{(1)}/k}(\log D^{(1)}))$ is canonically a direct summand of the de Rham cohomology 
$H^j(X, \Omega^{\bullet}_{X/k}(\log D)).$
\end{lemma}
\begin{proof} For $n \geq 0,$ recall the truncation $\tau_{< n} F_{X/k*} \Omega^{\bullet}_{X/k}(\log D),$ whose cohomology 
agrees with that of $F_{X/k*} \Omega^{\bullet}_{X/k}(\log D)$ in degrees $< n$ and is zero otherwise. 
By \cite[4.2.3]{DeligneIllusie}, the complex of $\O_{X^{(1)}}$-modules $\tau_{< p} F_{X/k*} \Omega^{\bullet}_{X/k}(\log D)$ is 
canonically (in a way which depends on the chosen lifting of $(X,D)$) quasi-isomorphic to the direct 
sum of its cohomology sheaves, so that the Cartier isomorphism induces an isomorphism in the derived category of $\O_{X^{(1)}}$-modules 
$$ \oplus_{i < p} \Omega^i_{X^{(1)}/k}(\log D^{(1)})[-i] \simeq  \tau_{< p}F_{X/k*} \Omega^{\bullet}_{X/k}(\log D).$$
Hence $H^0(X^{(1)}, \Omega^j_{X^{(1)}/k}(\log D^{(1)}))$ is a direct summand in 
$$H^j(X^{(1)}, \tau_{< p}F_{X/k*} \Omega^{\bullet}_{X/k}(\log D)) \simeq H^j(X, \tau_{< p} \Omega^{\bullet}_{X/k}(\log D))$$ 
which agrees with $H^j(X, \Omega^{\bullet}_{X/k}(\log D)),$ as $j < p$ and the cohomology of the cone of 
$\tau_{< p} \Omega^{\bullet}_{X/k}(\log D) \rightarrow \Omega^{\bullet}_{X/k}(\log D)$ vanishes in degree $< p.$
\end{proof}

\begin{para}\label{para:basicsetup} Let $K$ be a field of characteristic $0.$ By a {\em $p$-adic valuation} on $K,$ we mean a rank one valuation $v$ on $K,$ with $v(p) > 0.$ 
We suppose that $K$ is complete with respect to such a valuation, with ring of integers $\O_K$ and perfect residue field $k.$ 
We now recall the facts we will need about the prismatic cohomology of smooth formal schemes over $\O_K.$ 

We do not recall the general formalism of prisms, as developed by Bhatt-Scholze \cite{BhattScholze}, and Koshikawa \cite{KoshikawaI} in the logarithmic case, but consider only the two examples we will need. First suppose that the valuation on $K$ is discrete. 
Set $A = W(k)\lps u \rps,$ equipped with a Frobenius $\varphi$ which extends the Frobenius on $W$ by sending $u$ to $u^p.$  We equip $A$ with the map $A \rightarrow \O_K$ sending $u$ to some chosen uniformizer $\pi.$ Its kernel is generated by an Eisenstein polynomial $E(u) \in W(k)[u]$ for $\pi.$ Then $(A, E(u)A)$ is the so called {\it Breuil-Kisin} prism. In fact, in applications we will assume $\O_K = W(k),$ and $\pi = p.$  

Next suppose that $K$ is algebraically closed. Let $R = \ilim \O_K/p$ with the maps in the inverse limit being given by the absolute Frobenius. 
We take $A = A_{\inf} = W(R),$ with its canonical Frobenius. Any element $x = (x_0,x_1, \dots ) \in R$ lifts uniquely to a sequence $(\hat x_0, \hat x_1, \dots)$ in $\O_K$ with $\hat x_i^p = \hat x_{i-1}.$ There is a natural surjective map of rings $\theta: A \rightarrow \O_K,$ which sends the Teichm\"uller representative of an element $x$ as above to $\hat x_0.$ The kernel of $\theta$ is principal, generated by $\xi = p - [\underline p],$ where $\underline p = (p, p^{1/p}, \dots)$ for  some choice of these roots.  Then $(A,\xi A)$ is an example of a perfect prism.
\end{para} 

\begin{para} In the rest of this section we will make use of logarithmic formal schemes over 
$\O_K,$ logarithmic schemes over $K,$ and logarithmic adic spaces over $K.$ In particular, when $K$ is algebraically 
closed we will consider the \'etale cohomology of logarithmic schemes and adic spaces over $K$ with coefficients in $\F_p.$ 
We refer the reader to \cite{Katolog} for a general introduction to logarithmic geometry, and to \cite{IllusieOverview} and 
\cite{DLLZ} for \'etale cohomology in this context. 
\end{para}

\begin{para}\label{para:prismaticproperties}
Now assume that $K$ is either discretely valued or algebraically closed. Let $X$ be a formally smooth, $p$-adic formal $\O_K$-scheme, equipped with a relative normal crossings divisor $D$. We write $X_D$ for the formal scheme $X$ equipped with the logarithmic structure given by $D.$  We will denote by $X_{D,K}$ the associated logarithmic adic space. When the divisor $D$ is empty we drop it from the notation. 

By \cite{BhattScholze} and \cite{KoshikawaI}, the {\it prismatic cohomology} of $X_D$ is a complex of $A$-modules 
$R\Gamma_{\Prism}(X_D/A),$ equipped with a $\varphi$-semi-linear map endomorphism, which we again denote by $\varphi.$ 
Note that \cite{KoshikawaI}, considers a more general situation where $A$ is equipped with a possibly non-trivial log structure. 
Here we always consider the trivial log structure on $A.$
As we will only be interested in mod $p$ cohomology we set 
$$R\Gamma_{\bar \Prism}(X_D/A) = R\Gamma_{\Prism}(X_D/A)\otimes^{L}_AA/pA,$$ 
and we will denote by $H^i_{\bar \Prism}(X_D/A)$ the cohomology of $R\Gamma_{\bar \Prism}(X_D/A).$
Then we have the following properties, where  for a ring $B$ we denote by $D(B)$ derived category of $B$-modules.
\begin{enumerate} 
\item There is a canonical isomorphism of commutative algebras in $D(k).$ 
$$ R\Gamma(X_k, \Omega^{\bullet}_{X_k/k}(\log D_k)) \simeq R\Gamma_{\bar \Prism}(X_D/A)\otimes_{A/pA,\varphi}^L k.$$
\item If $K$ is algebraically closed then there is an isomorphism of commutative algebras in $D(\F_p)$
$$ R\Gamma_{\et}(X_{D,K}, \mathbb F_p) \simeq R\Gamma_{\bar \Prism}(X_D/A)[1/\xi]^{\varphi=1} $$
Here $X_{D,K}$ denotes the adic space $X_K$ associated to $X,$ equipped with the log structure given by $D_K,$ 
and the complex on the left is the cohomology of the Kummer \'etale site of $X_{D,K}$ which is denoted by 
$R\Gamma_{k\et}$ in \cite{KoshikawaYao}. 
The complex on the right is the fiber of $\varphi-1$ acting on $ R\Gamma_{\bar \Prism}(X_D/A)[1/\xi],$ 

\item Let $d$ be a generator of $\ker(A\rightarrow \O_K).$ The linearization 
$$ \varphi^*(R\Gamma_{\bar \Prism}(X_D/A)) \rightarrow R\Gamma_{\bar \Prism}(X_D/A)$$ 
becomes an isomorphism in $D(A/p)$ after inverting $d.$ 
For each $i \geq 0,$ there is a canonical map 
$$V_i: H^i_{\bar \Prism}(X_D/A) \rightarrow H^i(\varphi^*R\Gamma_{\bar \Prism}(X_D/A))$$ 
with $V_i\circ \varphi = \varphi\circ V_i = d^i.$
\item Let $K'$ be a field complete with respect to a $p$-adic valuation, and which is either discretely valued or algebraically closed. 
Let $A' \rightarrow \O_{K'}$ be the corresponding prism, as defined above. Suppose we are given map  of valued fields 
$K \rightarrow K',$ and a map $A \rightarrow A'$ compatible with the projections to $\O_K \rightarrow \O_{K'},$ and with Frobenius maps. 
Then there is a canonical isomorphism in $D(A'/p)$
$$ R\Gamma_{\bar \Prism}(X_D/A)\widehat \otimes^L_AA' \simeq R\Gamma_{\bar \Prism}(X_{D,\O_{K'}}/A') $$
Here $\widehat \otimes^L_A$ means we take the derived completion of the tensor product with respect to either the $u$-adic or $\xi$-adic topology 
on $A',$ depending on whether $K'$ is discretely valued or algebraically closed. 
We shall apply this when $K$ is discretely valued and $K'$ is algebraically closed.
When $X$ is proper one can simply take the tensor product without completing as one has the following:
\item When $X$ is proper over $\O_K,$ then $R\Gamma_{\bar \Prism}(X_D/A)$ is a perfect complex of $A/p$-modules.
\end{enumerate}

In the non-logarithmic case, when $D$ is empty, this is \cite[Thm.~1.8]{BhattScholze}. For the logarithmic case, see the recent work of 
Koshikawa-Yao \cite[Thm.~2, Rem.~6.3]{KoshikawaYao}. Actually (1), (4) and (5) were proved in the earlier work of Koshikawa \cite[Rmk. 5.6, Cor. 5.5, and Thm. 5.3 respectively]{KoshikawaI}.
Note that (5) follows from the Hodge-Tate comparison in these references, cf.~\cite[Example 1.6]{KoshikawaI}.

Using (5), when $X$ is proper, we have the following more explicit form of (2) above.
\end{para}

\begin{lemma}\label{lemma:etalecomparison} Suppose that $K$ is algebraically closed, and that $X$ is proper over $\O_K.$ 
Then for each $i \geq 0$ there is a natural isomorphism 
$$ H^i_{\et}(X_{D,K}, \mathbb F_p)\otimes_{\mathbb F_p} (A/pA) [1/\xi] \simeq H^i_{\bar \Prism}(X_D/A)[1/\xi] .$$
\end{lemma}
\begin{proof} By (2) above there is an exact sequence in cohomology 
$$ H^i_{\et}(X_{D,K}, \mathbb F_p) \rightarrow H^i_{\bar \Prism}(X_D/A)[1/\xi] \overset {1-\varphi} \rightarrow 
H^i_{\bar \Prism}(X_D/A)[1/\xi] \rightarrow H^{i+1}_{\et}(X_{D,K}, \mathbb F_p)  $$
Now $A/pA[1/\xi] = \text{\rm Frac} R$ is an algebraically closed field, and $H^i_{\bar \Prism}(X_D/A)[1/\xi]$ is a finite dimensional 
$A/pA[1/\xi]$ vector space by (5) above. Hence $1-\varphi$ is surjective on $H^i_{\bar \Prism}(X_D/A)[1/\xi],$ and this space is spanned by 
its Frobenius invariants \cite[Prop.~4.1.1]{Katz}.
\end{proof}

\begin{para}\label{para:derivedcomplete} We will need some facts about derived complete modules, which we now recall 
\cite[\S 3.4]{BhattScholzeproetale}. Keeping the notation of \ref{para:prismaticproperties}(3), a complex $C$ in $D(A/p)$ is called {\em derived complete} if the natural map 
$$C \rightarrow \Rlim_n (C\otimes^L_{A/p}A/(p,d^n)A)$$ 
is a quasi-isomorphism. An $A/pA$-module $M$ is called derived complete, if 
it is derived complete as a one term complex. A complex $C$ is derived complete if and only if its cohomology groups are 
\cite[Prop.~3.4.4]{BhattScholzeproetale}. 

Now suppose we are in the situation of \ref{para:prismaticproperties}(4), and for an $A/p$-module $M,$ 
write 
$$M\widehat \otimes^L_AA' = \Rlim_n M\otimes^L_{A/p}A'/(p,d^n)A'.$$
\end{para}

\begin{lemma}\label{lem:derivedcomplete} With the notation of  \ref{para:prismaticproperties}(4), suppose that $K$ is 
discretely valued. Then we have 
\begin{enumerate}
\item If $M$ is derived complete, then $M\widehat \otimes^L_AA'$ is concentrated in degree $0.$
\item The functor $M \mapsto H^0(M\widehat \otimes^L_AA')$ on derived complete $A/p$-modules is exact.
\item We have $H^i_{\bar\Prism}(X_{D,\O_{K'}}, A') \simeq H^0(H^i_{\bar\Prism}(X_{D,\O_{K}}, A)\widehat \otimes^L_AA')$
\end{enumerate}
\end{lemma}
\begin{proof} (2) is a formal consequence of (1). 
Note that \ref{para:prismaticproperties}(4) applied with $A'=A,$ implies that $H^i_{\bar\Prism}(X_{D,\O_{K}}, A)$ 
is derived complete. Hence (3) follows from (1) and \ref{para:prismaticproperties}(4). It remains to show (1).

Write $M_{A'} = M\otimes_AA'$ for the ordinary, non-completed, tensor product. 
We have 
$$ M\widehat \otimes^L_AA' = \Rlim_n(M_{A'} \overset {d^n} \rightarrow M_{A'}),$$ 
where the transition maps are given by the identity in degree $0,$ 
and multiplication by $d$ in degree $-1.$ Since $\Rlim^i_n = 0$ for $i > 1,$ and the transition maps on 
$H^0(M_{A'}  \overset {d^n} \rightarrow M_{A'} )$ 
are surjective, the cohomology of $M\widehat \otimes^L_AA'$ is concentrated in degrees $-1,0,$ and 
$$ H^{-1}(M\widehat \otimes^L_AA') = \lim_n M_{A'} [d^n] =  \lim_n M[d^n]\otimes_AA' = T_d(M_{A'} ),$$ the $d$-adic 
Tate module of $M_{A'}.$ Thus we have to show $T_d(M_{A'}) = 0.$ 

Since $M$ is derived complete, the same argument with $A$ in place of $A',$ shows that $T_d(M) = 0.$ 
If $m \in T_d(M_{A'}),$ write $m = (m_1, m_2, \dots)$ with $m_i \in M_{A'}[d^i].$ If $m \neq 0,$ then $m_{i_0} \neq 0,$ 
for some $i_0.$ Let $N \subset A'/p$ be a finitely generated, saturated $A/p$-submodule such that $m_{i_0} \in M[d^{i_0}]\otimes_AN.$ 
We claim that the inclusion of $A$-modules $N \subset A'/p$ admits a continuous splitting $f:A'/p \rightarrow N.$ 
Assuming this, we see that $f(m) \in T_d(M)\otimes_A N$ is non-zero, a contradiction. It remains to show the existence of $f.$

Choose a complement $\bar N'$ to the $k$-subspace 
$N\otimes_Ak \subset A'\otimes_Ak,$ and a $k$-basis $\{\bar e_j\}_{j \in J}$ for $\bar N'.$ Let $e_j \in A'/p$ be a lift of $\bar e_j,$ and 
let $N'$ denote the $d$-adic completion of the free $A/p$-module with basis indexed by a set of elements $\{\hat e_j\}_{j \in J}.$ 
Then there is a unique continuous map $h: N\oplus N' \rightarrow A'/p,$ which is the identity on $N,$ and sends $\hat e_i$ to $e_i.$ 
The $d$-completeness of the source of $h$, combined with the fact that $K$ is discretely valued, implies that $h$ is surjective. 
If $C = \ker(h)$ then, since $A'/p$ is $d$-torsion free, and $h\otimes_Ak$ is an isomorphism, we have $C/dC = 0.$ 
As $N\oplus N'$ is $d$-adically separated, this implies $C= 0,$ 
so $h$ is an isomorphism, and we may take $f$ to be the projection of $N\oplus N'$ onto $N'.$
\end{proof}

\begin{para} We do not know if Lemma \ref{lem:derivedcomplete} continues to hold without assuming $K$ is discretely valued.
\end{para}

\subsection{Restriction of mod $p$ \'etale cohomology} The goal of this subsection  is to prove a result on restriction of 
\'etale cohomology classes to open neighborhoods. We begin with a technical lemma. As above, $k$ denotes a perfect field of characteristic $p.$ 

\begin{lemma}\label{lemma:notorsion} Let $M$ be a finitely generated $k\lps u \rps$-module, equipped with a Frobenius semi-linear map 
$\varphi: M \rightarrow M,$ and a linear map $V: M \rightarrow \varphi^*M$ such that $\varphi\circ V = V\circ \varphi = u^d$ for some integer $d.$ 
If $d+1 < p,$ then $M$ is torsion free.
\end{lemma} 
\begin{proof} Suppose that $M$ contains nontrivial torsion. Since $M$ is finitely generated over $k\lps u \rps,$ there is a minimal integer 
$m \geq 1$ such that $u^m\cdot x = 0$ for any torsion element $x \in M.$ Choose a torsion element $x$ so that $u^{m-1}\cdot x \neq 0.$ 
Then in $\varphi^*M$ we have 
$$ u^{d+m}(1\otimes x) = V\circ (1\otimes\varphi)(u^m\otimes x) = V(u^m\varphi(x)) = 0,$$ 
since $u^{pm}\varphi(x) = 0,$ so that $\varphi(x)$ is torsion in $M.$ 
On the other hand, as $\varphi$ is finite flat over $k\lps u \rps,$ the smallest power of $u$ which kills 
$1\otimes x \in \varphi^* M$ is $pm.$ 
Thus $d + m \geq pm,$ which implies $p \leq 1 + d/m \leq 1+d.$
\end{proof}

\begin{para} 
Now let $K$ be as in \ref{para:basicsetup} and suppose that $K$ is discretely valued, and that $\O_K = W(k).$ 
Let $C$ be an algebraically closed field equipped with a complete $p$-adic valuation, and  $K \subset C$ an inclusion of valued fields. 
 As before, for a formal scheme or formal log~scheme $X$ over $\O_K,$ we denote by $X_K$ and $X_C$ the associated adic spaces 
over $K$ and $C$ respectively. We denote by $X_{\O_C} = X\hotimes_{\O_K}\O_C$ the base change of $X$ to $\O_C,$ as a formal 
scheme, and by $X_k$ the special fiber (i.e~the reduced subscheme) of $X.$ 

Unless otherwise indicated, for the rest of this subsection, we let $X$ be a proper, smooth formal scheme over $\O_K,$ 
equipped with a relative normal crossings divisor $D \subset X.$   Let
\[h^{0,i}_{X,D}:= \dim_K H^0(X_K,\Omega^i_{X_K/K}(\log D)).\]
\end{para}

\begin{prop}\label{prop:keyrestrict} Let $W \subset X-D$ be a dense open formal subscheme.  Then for $0 \leq i < p-2,$ 
$$ \dim_{\mathbb F_p}\Im\!(H^i_{\et}(X_{D,C}, \mathbb F_p) \rightarrow H^i_{\et}(W_C, \mathbb F_p)) \geq  h^{0,i}_{X,D}.$$
\end{prop}
\begin{proof} Let $k_C$ be the residue field of $C.$ We may replace $K$ by $W(k_C)[1/p]$ and assume that $k_C = k.$ 
As $W \subset X-D,$ we will omit the divisor $D$ from the notation when writing the cohomology of $W.$

Set 
$$ M_{\Prism} = \Im(H^i_{\bar \Prism}(X_D/A) \rightarrow H^i_{\bar \Prism}(W/A)),$$ 
which is a finitely generated $A/pA = k\lps u \rps$-module. Using  \ref{para:prismaticproperties}(3), one sees that there are maps 
$$ M_{\Prism} \overset V \rightarrow \varphi^* M_{\Prism} \overset \varphi \rightarrow M_{\Prism} $$
with $\varphi\circ V = V\circ \varphi = u^i.$ Hence Lemma \ref{lemma:notorsion} and our assumptions on $i$ imply that 
$M_{\Prism}$ is a finitely generated, free $k\lps u \rps$-module, and thus derived complete. It follows that 
$$ M_{\Prism,R} : = M_{\Prism}\otimes_{k\lps u \rps} R \subset H^i_{\bar \Prism}(W_{\O_C}/A_C)$$
is a free $R$-module, where the inclusion follows from Lemma \ref{lem:derivedcomplete}.

By \ref{para:prismaticproperties}(4), there is an isomorphism 
$$ H^i_{\bar \Prism}(X_D/A)\otimes_AA_C \iso H^i_{\bar\Prism}(X_{D,\O_C}/A_C).$$ 
Thus using Lemma \ref{lemma:etalecomparison} we have maps 
 \begin{multline}  H^i_{\bar \Prism}(X_D/A)\otimes_AA_C[1/\xi] \simeq H^i_{\et}(X_{D,C}, \mathbb F_p)\otimes_{\F_p}R[1/\xi] \\
\rightarrow H^i_{\et}(W_C, \mathbb F_p)\otimes_{\F_p}R[1/\xi] \rightarrow  H^i_{\bar \Prism}(W_{\O_C}/A_C)[1/\xi], 
\end{multline}
the composite being the natural map. 
Hence 
$$\dim_{\mathbb F_p}\Im (H^i_{\et}(X_{D,C}, \mathbb F_p) \rightarrow H^i_{\et}(W_C, \mathbb F_p)) \geq \dim_{R[1/u]} M_{\Prism,R}[1/u].$$
As $M_{\Prism}$ finite free over $k\lps u \rps,$ it suffices to show that $\dim_k M_{\Prism}/uM_{\Prism} \geq h^{0,i}_{X,D}.$

Using Lemma \ref{lemma:notorsion} again, we see that $H^j_{\bar\Prism}(X_D/A)$ is $u$-torsion free for $0 \leq j \leq i+1.$ 
This torsion freeness for $j=i,i+1$ combines  with \ref{para:prismaticproperties}(1) to give an  isomorphism
\[
	H^i(X_k, \Omega_{X_k/k}^{\bullet}(\log D_k)) \simeq H^i_{\bar \Prism}(X_D/A)\otimes_{A/pA,\varphi} k 
\]
and thus a map
   \begin{multline}\label{deRmap}  H^i(X_k, \Omega_{X_k/k}^{\bullet}(\log D_k)) \simeq H^i_{\bar \Prism}(X_D/A)\otimes_{A/pA,\varphi} k 
\rightarrow M_{\Prism}\otimes_{A/pA,\varphi}k \\
\rightarrow H^i_{\bar \Prism}(W_{\O_C}/A_C)\otimes_{A_C/pA_C,\varphi} k \rightarrow H^i(W_k, \Omega_{W_k/k}^{\bullet}),
\end{multline} 
where the composite is the natural map. 
This shows the image of (\ref{deRmap}) has dimension $\leq \dim_k M_{\Prism}/uM_{\Prism},$ and it suffices to show 
that the dimension of the image is $\geq  h^{0,i}_{X,D}.$ Since $W \subset X$ is dense, the map 
$$ H^0(X_k, \Omega^i_{X_k/k}(\log D_k)) \rightarrow H^0(W_k, \Omega^i_{W_k/k}),$$
is injective. Hence by Lemma \ref{lem:degdeR}, the image of (\ref{deRmap}) has dimension at least $\dim_k H^0(X_k, \Omega^i_{X_k/k}(\log D_k)) \geq h^{0,i}_{X,D},$ where the last inequality follows from the  upper semi-continuity  of $h^{0,i}$.
\end{proof}

\begin{cor}\label{cor:keyrestrict} Suppose that $H^{i+1}_{\bar \Prism}(X_D/A)$ 
is $u$-torsion free. Then the conclusion of Proposition \ref{prop:keyrestrict} holds for $i \leq p-2.$ In particular, the conclusion holds for $i \leq p-2$ 
if $X$ is the formal completion of an abelian scheme over $\O_K,$ and $D$ is empty.
\end{cor}
\begin{proof} The stronger assumption $i+1 < p-1$ was used in the proof of Proposition \ref{prop:keyrestrict} only to know that 
$H^{i+1}_{\bar\Prism}(X_D/A)$ is $u$-torsion free, so the first claim follows. For the second claim, we remark that Ansch\"utz-Le Bras \cite[Prop.~4.58, Cor.~4.64]{AnschutzLeBras} have shown that 
when $X$ is the formal completion of an abelian scheme, the ring $H^\bullet_{\bar \Prism}(X/A)$ is the exterior algebra on 
the $k\lps u \rps$-module $H^1_{\bar \Prism}(X/A)$ 
which is free of rank $2g = 2\dim X_K$ over $k\lps u \rps.$ 
\end{proof}

\begin{para}\label{para:albHN}  
Our next goal is to explain a variant of Proposition \ref{prop:keyrestrict} which will be useful in obtaining lower bounds on the 
essential dimension of covers whose groups are elementary $p$-groups. Let 
$$ h^{0,1\shortrightarrow i}_{X,D} : = \dim_K(\Im\!\!(\wedge^i H^0(X_K,\Omega^1_{X_K/K}(\log D)) \rightarrow H^0(X_K,\Omega^i_{X_K/K}(\log D))),$$
and 
$$ h^{0,1\shortrightarrow i}_{X_k,D_k} : = \dim_k(\Im\!\!(\wedge^i H^0(X_k,\Omega^1_{X_k/k}(\log D)) \rightarrow H^0(X_k,\Omega^i_{X_k/k}(\log D))).$$
\end{para}

\begin{prop}\label{prop:keyrestrict2} Let $W \subset X-D$ be a dense open formal subscheme, and $p > \max\{i+1, 3\}.$ 
Then 
$$ \dim_{\mathbb F_p}\Im (\wedge^i H^1_{\et}(X_{D,C}, \mathbb F_p) \rightarrow H^i_{\et}(W_C, \mathbb F_p)) 
\geq   h^{0,1\shortrightarrow i}_{X_k,D_k}, $$ 
where the map is given by the cup product followed by restriction of classes to $W_C.$
\end{prop}
\begin{proof} As the proof is analogous to that of Proposition \ref{prop:keyrestrict}, we only sketch it. 
Let 
$$ M_{\Prism} = \Im(\wedge^i H^1_{\bar\Prism}(X_D/A) \rightarrow H^i_{\bar\Prism}(W/A)).$$
Arguing as in Proposition \ref{prop:keyrestrict}, since $i < p-1,$ one sees that it suffices to show 
$\dim_k M_{\Prism}/uM_{\Prism} \geq h^{0,1\shortrightarrow i}_{X_k,D_k}.$ 
Since $p > 3,$ $H^2_{\bar \Prism}(X_D/A)$ is $u$-torsion free by Lemma \ref{lemma:notorsion}, and 
we reduce to showing that the image of 
 $$  \wedge^i H^1(X_k, \Omega_{X_k/k}^{\bullet}(\log D_k)) \rightarrow H^i(W_k, \Omega_{W_k/k}^{\bullet}), $$
has dimension $ \geq h^{0,1\shortrightarrow i}_{X_k,D_k}.$ 
It then suffices to check that the image of the composite 
$$  \wedge^i H^0(X_k, \Omega^1_{X_k/k}(\log D_k)) \rightarrow H^0(X_k, \Omega^i_{X_k/k}(\log D_k)) \rightarrow H^0(W_k, \Omega^i_{W_k/k})$$ 
has dimension $ \geq h^{0,1\shortrightarrow i}_{X_k,D_k}.$ 
By definition, the image of the first map has dimension $ h^{0,1\shortrightarrow i}_{X_k,D_k},$ and the second map is injective, 
so the result follows. 
\end{proof}

\begin{para}\label{para:schemesetup} We now want to deduce analogues of the above results for schemes. 
Let $Y$ be a smooth scheme over $C,$ equipped with a normal crossings divisor $D \subset Y.$ 
As above, we write $h^{0,i}_{Y,D} = \dim_C H^0(Y,\Omega^i_{Y/C}(\log D)).$

We will denote by $Y^{\ad}$ and $D^{\ad}$ the adic spaces attached to $Y,$ and $D$ respectively. Then $(Y^{\ad}, D^{\ad})$ is a logarithmic adic space. We begin by recording 
a lemma comparing the \'etale cohomology of logarithmic adic spaces and logarithmic schemes. 
\end{para}

\begin{lem}\label{etalecomparison} Let $Y$ be a proper smooth scheme over $C,$ 
equipped with a normal crossings divisor $D \subset Y,$ and set  $U = Y-D.$
Then for $i \geq 0$ there are natural isomorphisms 
$$ H^i_{\et}(Y^{\ad}_{D^{\ad}}, \F_p) \simeq H^i_{\et}(Y_D, \F_p) \simeq H^i_{\et}(U, \F_p) $$
\end{lem}
\begin{proof} For second isomorphism see \cite[Thm.~7.4]{IllusieOverview}. 
The first isomorphism follows from \cite[Thm~7.2, Cor~7.5]{IllusieOverview}, \cite[Lem.~4.6.2]{DLLZ} and 
\cite[Prop.~2.1.4, Thm~3.8.1]{Huber}. 
More precisely, let $\varepsilon: Y_{D,k\et} \rightarrow Y_{\et}$ be the natural morphism of sites, and 
set $M = R\varepsilon_* \F_p.$ Then 
the first two references show that 
\begin{itemize}
\item $H^i_{\et}(Y,M)$ is naturally isomorphic to $H^i_{\et}(Y_D, \F_p).$ 
\item If $M^{\ad}$ denotes the pullback of $M$ to $Y^{\ad}_{\et},$ then 
$H^i_{\et}(Y^{\ad},M^{\ad})$ is naturally isomorphic to $H^i_{\et}(Y^{\ad}_{D^{\ad}}, \F_p).$
\end{itemize}
The two results in \cite{Huber}  then provide a natural isomorphism $H^i_{\et}(Y,M) \simeq H^i_{\et}(Y^{\ad},M^{\ad}).$
\end{proof}

\begin{prop}\label{prop:restrictschem} Let $X$ be a proper smooth scheme over $\O_K,$ equipped with a relative normal crossings divisor $D \subset X.$ Set $U = X\backslash D,$  and let $W \subset U_C$ be a dense open subscheme. 
If $0 \leq i < p-2,$ then 
$$ \dim_{\mathbb F_p}\Im (H^i_{\et}(U_C, \mathbb F_p) \rightarrow H^i_{\et}(W, \mathbb F_p)) \geq  h^{0,i}_{X_C,D_C}.$$
If $X$ is an abelian scheme and $D$ is empty, then the same statement holds for $i \leq p-2.$

If $p > \max\{i+1, 3\}$ then 
$$ \dim_{\mathbb F_p}\Im (\wedge^i H^1_{\et}(U_C, \mathbb F_p) \rightarrow H^i_{\et}(W, \mathbb F_p)) 
\geq  h^{0,1\shortrightarrow i}_{X_k,D_k}.$$
\end{prop}
\begin{proof} Let $k_C$ be the residue field of $C.$ We may replace $X$ by its base change to $W(k_C),$ and assume that $C$ and $K$ have the same residue field. Denote by $\widehat X$ and $\widehat D$ the formal completions of $X$ and $D.$ 
Let $Z$ be the closure of $X_C - W$ in $X_{\O_C},$ and 
let $\widehat W \subset \widehat X$ be the formal open subscheme, which is the complement of $Z\otimes k$ in $\widehat X.$ 
Note that we have $\widehat W_C \subset W^{\ad}.$ 
Thus we have a commutative diagram of natural maps 
$$ \xymatrix{H^i_{\et}(X_{D,C}, \mathbb F_p)\ar[r]\ar[dd]^\sim & H^i_{\et}(W, \mathbb F_p) \ar[d] \\
  & H^i_{\et}(W^{\ad}, \mathbb F_p) \ar[d] \\
 H^i_{\et}(\widehat X_{D,C}, \mathbb F_p) \ar[r] & H^i_{\et}(\widehat W_C, \mathbb F_p) 
} $$
where the isomorphism on the left is given by Lemma \ref{etalecomparison}. 
By Propositions \ref{prop:keyrestrict},  \ref{prop:keyrestrict2}  and Corollary \ref{cor:keyrestrict}, we have the inequalities claimed 
in the proposition, but for the dimension of the images of $H^i_{\et}(\widehat X_{D,C}, \mathbb F_p)$ 
and $\wedge^i H^1_{\et}(\widehat X_{D,C}, \mathbb F_p)$ in $H^i_{\et}(\widehat W_C, \mathbb F_p).$ 
The proposition now follows from the commutative diagram above, as well as Lemma \ref{etalecomparison}.
\end{proof}

\begin{para}\label{para:goodredn} The previous results apply for schemes that are smooth over $\O_K.$ We now want to reformulate these results to show that they hold for any algebraically closed field of characteristic 0, for $p\gg 0.$ Thus we now assume that $C$ is algebraically closed of characteristic 0, but we no longer assume it is equipped with a complete  $p$-adic valuation.

Let $Y$ be a proper, smooth scheme over $C,$ and $D \subset Y$ a normal crossings divisor. 
We say that $(Y,D)$ has {\it good reduction at $p$} if there exists 
a $p$-adic valuation on $C$ (which we do not assume complete), with ring of integers $\O_C,$ and $Y$ extends to a smooth proper $\O_C$-scheme $Y^\circ$ with a relative normal crossings divisor $D^\circ \subset Y^\circ$ over $\O_C,$ extending $D.$ 
We say that $(Y,D)$ has {\it unramified good reduction at $p$} if in addition $(Y^{\circ}, D^{\circ})$ can be chosen so that it descends to an absolutely unramified discrete valuation ring (with respect to the given valuation) $\O \subset \O_C.$ 

Let $k$ be the residue field of $\O.$ Then we have the invariants $h^{0,1\shortrightarrow i}_{Y^\circ_k,D^\circ_k}$ defined 
as in \ref{para:albHN}. We set  $h^{0,1\shortrightarrow i}_{Y,D,p} = \max h^{0,1\shortrightarrow i}_{Y^\circ_k,D^\circ_k}$ 
with the maximum taken over all choices of $(Y^{\circ}, D^{\circ})$ as above.
\end{para}

\begin{cor}\label{cor:goodredn} Let $Y$ be a proper, smooth scheme over $C,$ $D \subset Y$ a normal crossings divisor, and  $W \subset U = Y - D$ a dense open subscheme. 
Suppose that $(Y,D)$ has unramified good reduction at $p.$ If $0 \leq i < p-2,$ then 
$$ \dim_{\mathbb F_p}\Im (H^i_{\et}(U, \mathbb F_p) \rightarrow H^i_{\et}(W, \mathbb F_p)) \geq h^{0,i}_{Y,D}.$$
If $Y$ is an abelian scheme, then the same statement holds for $i \leq p-2.$

If $p > \max\{i+1, 3\}$ then 
$$ \dim_{\mathbb F_p}\Im (\wedge^i H^1_{\et}(U, \mathbb F_p) \rightarrow H^i_{\et}(W, \mathbb F_p)) 
\geq  h^{0,1\shortrightarrow i}_{Y,D,p}. $$
\end{cor}
\begin{proof}  If $C \subset C'$ is any algebraically closed field, then the \'etale cohomology groups in the corollary do not change if we replace 
$(Y,D)$ and $W$ by their base change to $C'.$ Thus we may assume that $C$ is complete. We may then also assume that $\O \subset \O_C$ 
in \ref{para:goodredn} is $p$-adically complete, and unramified. The result now follows from Proposition \ref{prop:restrictschem}.
\end{proof}

\begin{para}\label{para:genericgoodredn} We remark that when $(Y,D)$ is a generic member of a good moduli space, then $(Y,D)$ has unramified good reduction at {\it all} primes, and this condition in Corollary \ref{cor:goodredn} can 
then be suppressed.

More precisely, suppose $\mathcal Y \rightarrow \M$ is proper smooth, with $\M$ a smooth, faithfully flat, connected, separated 
Deligne-Mumford stack over $\mathbb Z$, as in \cite[Definition 4.1, 4.14]{LaumonMoretBailly}. 
Let $\mathcal D \subset \mathcal Y$ be a relative normal crossing divisor over $\M.$ 
If $\Spec C \rightarrow \M$ is a map whose image is the generic point $\eta \in \M,$ then $(Y,D) = (\mathcal Y\times_{\M} \Spec C, \mathcal D\times_{\M} \Spec C)$ has unramified good reduction at all primes. 
\footnote{Note that we are {\em not} asserting that all points in an open substack have good unramified reduction at all primes. 
Indeed, unless  $\mathcal D$ is empty, there will be no such open substack even if we ask for good unramified reduction at a single prime.}

To see this, note that our assumptions on $\M$ imply that there exists an \'etale presentation $\tilde \M \rightarrow \M,$ 
with $\tilde \M$ a scheme which is smooth and faithfully flat over $\mathbb Z.$ Hence $\eta$ lifts to 
 a generic point $\tilde \eta \in \tilde \M(C),$ which admits a unique specialization to a characteristic $p$ point $\bar{\tilde \eta} \in \M.$ 
As $\tilde \M$ is smooth over $\Z,$ the local ring $\O_{\tilde \M,\bar{\tilde\eta}}$ is a discrete valuation ring with uniformizer $p$, and this induces the required unramified $p$-adic valuation on $C.$
This discussion applies, for example, to the universal family of principally polarized abelian varieties or curves. 
\end{para}

\begin{cor}\label{cor:arithmeticrestr} Let $Y$ be a proper, smooth scheme over $C,$ $D \subset Y$ a normal crossings divisor, 
$U = Y-D,$ and  $ \eta \in Y$ the generic point. Then for $i \geq 0,$ and  $p \gg 0,$  
$$ \dim_{\mathbb F_p}\Im (H^i_{\et}(U, \mathbb F_p) \rightarrow H^i_{\et}(\eta, \mathbb F_p)) \geq  h^{0,i}_{Y,D}. $$
and 
$$ \dim_{\mathbb F_p}\Im (\wedge^i H^1_{\et}(U, \mathbb F_p) \rightarrow H^i_{\et}(\eta, \mathbb F_p)) \geq 
h^{0,1\shortrightarrow i}_{Y,D,p}. $$
\end{cor}
\begin{proof} By Corollary \ref{cor:goodredn}, it suffices to show that $(Y,D)$ has unramified good reduction at $p$ for sufficiently large $p.$ 
Since $Y$ is of finite type, there exists finite type $\mathbb Z$-algebra $\A \subset C$ such that $(Y,D)$ descends to 
a proper smooth $\A$-scheme $Y_{\A}$ equipped with a relative normal crossings divisor $D_{\A} \subset Y_{\A}.$ 

We may replace $\A$ by its normalization, and assume it is normal. Since $\A\otimes \Q$ is reduced, it is geometrically reduced. 
Hence for $p\gg 0,$ the $\F_p$-algebra $\A/p\A$ is reduced \cite[IV, Prop.~4.6.1, Thm.~9.7.7]{EGA}. 
Since $\A$ is of finite type over $\mathbb Z,$ for 
$p\gg 0,$ $p$ is the image of a prime $\mathfrak p \in \Spec \A$ of height $1.$ Fix $p$ such that $\A/p\A$ is reduced, 
and $\mathfrak p$ exists. Since $\A/p\A$ is reduced, $\A_{\mathfrak p}$ is a discrete valuation ring, with uniformizer $p.$ 
Extend the corresponding valuation $v_{\mathfrak p}$ to $C.$ Then $(Y,D)$ descends to $(Y_{\A}, D_{\A}),$ so $(Y,D)$ has 
unramified good reduction at $p.$
\end{proof}

\subsection{Characteristic classes}\label{subsec:characlasses} We continue to denote by $C$ an algebraically closed field of characteristic $0.$ 

\begin{para} Let $X$ be a proper, connected, smooth $C$-scheme, 
equipped with a normal crossings divisor $D,$ and let $U = X - D.$ 
We fix a geometric point $\bar \eta$ mapping to the generic point $\eta \in X.$ Let $G$ be a finite quotient of $\pi_{1,\et}(U,\bar\eta).$ 
For any $i$ there are canonical maps 
\begin{equation}\label{eqn:inflation}
H^i(G, \F_p) \rightarrow H^i(\pi_{1,\et}(U,\bar\eta), \F_p) \rightarrow H^i_{\et}(U,\F_p),
\end{equation}
where the first map is inflation of classes from $G$ to $\pi_{1,\et}(U,\bar\eta),$ and the second map is induced by considering 
 $\pi_{1,\et}(U,\bar\eta)$-representations as \'etale sheaves on $U.$
In the proposition below we will consider the condition that the composite of the two maps above is surjective. 
Note that since  $\pi_{1,\et}(U,\bar\eta)$ is topologically finitely generated, we can always choose $G$ such that 
the {\em first map} is surjective. In particular, 
if $U$ is an \'etale $K(\pi,1)$ we can choose $G$ so that the composite is surjective, but this is not the case in general. 

Finally for a $G$-cover $Y \rightarrow U,$ let $\ed(Y/U;p)$ denote the {\em$ p$-essential dimension} of $Y$ over $U$ 
\cite{ReiYo}, i.e. $\ed(Y/U;p)$  is the least $d$ for which there exists:
\begin{enumerate}
		\item a dense Zariski open $V\subset U$, 
		\item a finite  \'etale map $\pi\colon E\to V$ with $p\nmid \deg(\pi)$,
		\item a morphism $f\colon E\to Z$ with $\dim Z=d$, and
		\item a $G$-cover $\tilde{Z}\to Z$ for which $f^*\tilde{Z}\simeq Y\times_U E$.
\end{enumerate}
\end{para}

\begin{prop}\label{prop:essdimlowerbounds} Suppose that $i < p-2$ 
and that $(X,D)$ has unramified good reduction at $p.$ Let $G$ be a finite group  and let $Y \rightarrow U$ be a $G$-cover. 
Suppose that $h^{0,i}_{X,D} \neq 0$ and that the (restriction of the) classifying map 
\begin{equation}\label{eqn:charclassmap}
H^i(G, \F_p) \rightarrow H^i_{\et}(U,\F_p) 
\end{equation} 
is surjective.  Then $\ed(Y/U;p) \geq i.$  If $X$ is an abelian variety and $D=\emptyset$, the above holds for $i\leq p-2$.
\end{prop}

\begin{proof} Let $U' \rightarrow U$ be a finite, connected covering which has prime to $p$ degree over $\eta,$ and 
let $\eta' \in U'$ be the generic point. We have to show that $\ed(Y'/U') \geq i,$ where $Y' = Y\times_UU'.$
Consider the composite map 
$$ H^i(G, \F_p) \rightarrow H^i(\pi_{1,\et}(U,\bar\eta), \F_p) \rightarrow H^i_{\et}(U,\F_p) \rightarrow H^i(\eta, \F_p) \rightarrow H^i(\eta',\F_p).$$
Our assumptions imply that the composite of the first two maps is surjective. By Corollary \ref{cor:goodredn}, 
the third map is nonzero, as $h^{0,i}_{X,D} \neq 0.$ Thus the composite of the first three maps is nonzero. 
 The composite of the fourth map and the  trace map $H^i(\eta',\F_p) \rightarrow H^i(\eta,\F_p)$ is multiplication by $\deg(\eta'/\eta),$ 
 hence injective, as $\eta'/\eta$ has degree prime to $p.$ Thus the final map is an injection and the composite of all four maps is nonzero.
 
Suppose $\ed(Y'/U') < i.$ Then for some dense open $W \subset U',$ there is a map of $C$-schemes $W \rightarrow Z,$ with 
$\dim Z < i,$ and a $G$-cover $Y'_Z \rightarrow Z,$ such that $Y'|_W \simeq Y'_Z \times_ZW$ as $W$-schemes with $G$-action \cite[2.1.4]{FKW}. 
Shrinking $Z$ and $W$ if necessary, we may assume that $Z$ is affine. The above constructions, then give us a commutative diagram 
 
$$\xymatrix{
H^i(G, \F_p) \ar[r]\ar@{=}[d] & H^i_{\et}(Z, \F_p) \ar[d] \\
H^i(G, \F_p) \ar[r] & H^i_{\et}(W, \F_p) \ar[r] & H^i_{\et}(\eta', \F_p)
}$$

Since $Z$ is affine of dimension $< i$ it follows that  $H^i_{\et}(Z, \F_p) = 0$.  This 
implies that the composite of the maps in the bottom 
row is $0.$ This contradicts what we saw above.
\end{proof}

\begin{cor}\label{cor:abvar} Let $X/C$ be an abelian variety of dimension $g.$ Let $p \geq g+2,$ and suppose that $X$ 
has unramified good reduction at $p.$ Let $X' = X,$ viewed as a 
$(\Z/p\Z)^{2g}$-cover of $X$ via the multiplication by $p$ 
map $X' \rightarrow X.$

Then $\ed(X'/X;p)=g.$ In particular, this equality holds for $p\gg 0.$ 
\end{cor}

\begin{rem}
Corollary~\ref{cor:abvar} resolves almost all of a conjecture of Brosnan \cite[Conj. 6.1]{FS}. Note that \cite[Cor. 6.7]{FS} establishes Brosnan's conjecture for $\dim X\le 3$ and a positive density set of primes (depending on $X$), and one can also deduce the conjecture for a sufficiently generic abelian variety from \cite[Lemma 6.2]{FS}. Prior work of Gabber \cite{CT} established the result for a very general product of elliptic curves.
\end{rem}
\begin{proof}[Proof of Corollary~\ref{cor:abvar}] By definition, $g=\dim X \ge \ed(X'/X;p)$, so it suffices to prove that $\ed(X'/X;p)\ge g$. Let $G = (\Z/p\Z)^{2g}$ be the quotient of $\pi_{1,\et}(X,\bar\eta)$ corresponding to $X' \rightarrow X.$ 
Note that, in our present situation,  the map (\ref{eqn:charclassmap}) is surjective  because 
$H^\bullet_{\et}(X,\F_p)$ is the exterior algebra on $H^1_{\et}(X,\F_p).$ Since $h^{0,g}_X = 1,$ the inequality $\ed(X'/X;p) \geq g$ follows from Proposition \ref{prop:essdimlowerbounds}.

The final claim follows as in the proof of Corollary \ref{cor:arithmeticrestr}, as $X$ has unramified good reduction at all sufficiently large primes $p.$
\end{proof}

\begin{para} We now explain a generalization of Corollary \ref{cor:abvar}, for the mod $p$ homology cover $Y \rightarrow U.$ 
Recall that this is the cover corresponding to the maximal quotient of 
$\pi_{1,\et}(U, \bar\eta)$ which is an elementary abelian $p$-group. 
When $U$ is proper, this is just the pullback to $U$ of the cover described in Corollary \ref{cor:abvar} applied to the Albanese variety of $U.$
We define the invariant $h^{0,1\shortrightarrow i}_{X,D}$ as in \ref{para:albHN}. 
\end{para}

\begin{thm}\label{thm:ellabessdimlowerbounds}  Suppose $(X,D)$ has unramified good reduction at $p,$ 
and that $p > \max\{\dim X +1,3\}.$ 
Then  the mod $p$ homology cover $Y \rightarrow U$ satisfies 
$$ \ed(Y/U; p) \geq \max \{ i: h^{0,1\shortrightarrow i}_{X,D,p} > 0 \}. $$
In particular, if $p \gg 0$ then 
$$ \ed(Y/U; p) \geq  \max \{ i: h^{0,1\shortrightarrow i}_{X,D} > 0 \}. $$
\end{thm}
\begin{proof} As in the proof of Proposition \ref{prop:essdimlowerbounds}, let 
$U' \rightarrow U$ be a finite, connected covering which has prime to $p$ degree over $\eta,$ and let $\eta' \in U'$ be the generic point. 
Let $G = \Gal(Y/U),$ choose $i$ such that $h^{0,1\shortrightarrow i}_{X,D,p} > 0,$ and consider the composite map 
$$ \wedge^i H^1(G, \F_p) \simeq \wedge^i H^1_{\et}(U,\F_p) \rightarrow H^i(\eta, \F_p) \rightarrow H^i(\eta',\F_p).$$
By Corollary \ref{cor:goodredn}, the second map is nonzero, and the last map is injective as $U'$ has degree prime to $p$ over $\eta.$ 
As the composite map factors through $H^i(G,\F_p),$ it follows that 
$$ H^i(G,\F_p) \rightarrow H^i(\eta', \F_p)$$ 
is nonzero, which implies that $\ed(Y/U; p) \geq i,$ as in the proof of Proposition \ref{prop:essdimlowerbounds}.

This proves the first statement, and the second statement follows as for each $i,$  
$ h^{0,1\shortrightarrow i}_{X,D} = h^{0,1\shortrightarrow i}_{X,D,p}$ for $p \gg 0.$
\end{proof}

\begin{para} Suppose that $D$ is empty. Then the quantity $\max \{ i: h^{0,1\shortrightarrow i}_{X,D} > 0 \}$ which appears 
in Theorem \ref{thm:ellabessdimlowerbounds}, is the dimension of the image of $X$ under the Albanese map 
$\alpha_X: X \rightarrow \Alb(X).$ This is called the {\em Albanese dimension} of $U = X.$ 
To see this, note that $H^1(X,\Omega^1_{X/C})$ is isomorphic 
to the global $1$-forms on $\Alb(X),$ that these generate the global $i$-forms on $\Alb(X),$ 
and that $X$ is generically smooth over its image in $\Alb(X),$ as $\text{\rm char}(C) = 0$ 
(cf.~the argument in Proposition \ref{prop:albint} below). 

The following Proposition makes the result for $p \gg 0$ in Theorem \ref{thm:ellabessdimlowerbounds}, effective when $D$ is empty.  
\end{para}

\begin{prop}\label{prop:albint} Let $p$ be a prime of good unramified reduction for $X,$ 
so that $X$ arises from a smooth proper scheme $X^\circ$ over a valuation ring $\O_C$ as in \ref{para:goodredn}. Then 
\begin{enumerate}
\item $\alpha_X$ extends to a map $\alpha_X^\circ: X^\circ \rightarrow \Alb(X)^\circ$ over $\O_C,$ 
with $\Alb(X)^\circ$ an abelian scheme.
\item If the special fiber of $X^\circ$ is generically smooth over its image under $\alpha_X^\circ,$ then 
$$ \ed(Y/U; p) \geq  \dim \, \alpha_X(X),$$
where $Y\to U$ denotes the mod $p$ homology cover as above.
\end{enumerate}
\end{prop}
\begin{proof} Let $A = \Alb(X)^\vee,$ the dual abelian variety. Then $A \simeq \Pic^0_X,$ and any point of $X(C)$ gives rise to 
 a Poincar\'e line bundle $\L \rightarrow X \times A.$
By \cite[Thm.~2]{Koizumi},  $\Alb(X)$ and $A$ extend to abelian schemes $\Alb(X)^\circ$ and $A^\circ$ over $\O_C.$ 
More precisely, $X^\circ$ descends to a discrete valuation ring $\O$ as in \ref{para:goodredn}, and one may apply {\em loc.~cit} 
to this descent. 

As $X^\circ \times A^\circ$  is regular, $\L$ extends to a line bundle $\L^\circ \rightarrow X^\circ \times A^\circ,$ which corresponds to a map $$ X^\circ \rightarrow \Pic^0_{A^\circ/\O_C} \simeq A^{\circ \vee} \simeq \Alb(X)^\circ, $$
where $A^{\circ \vee}$ denotes the dual abelian scheme of $A^{\circ}.$ This proves the first claim.

As before denote by $k$ the residue field of $\O_C,$ and let $i = \dim \alpha_X(X).$ 
Suppose that the special fiber of $X^\circ$ is generically smooth over its image under $\alpha_X^\circ.$ 
Let $\eta$ be the generic point of $\alpha^{\circ}_X(X^\circ_k),$ and $\kappa(\eta)$ the residue field at $\eta.$ 
Then 
$$ \wedge^i H^0(A^{\circ\vee}, \Omega^1_{A^{\circ\vee}_k/k}) \simeq H^0(A^{\circ\vee}, \Omega^i_{A^{\circ\vee}_k/k}) $$
generates the $\kappa(\eta)$-vector space $\Omega^i_{\kappa(\eta)/k} \neq 0.$ As $X^\circ_k$ is smooth over $\eta,$ 
$\Omega^i_{\kappa(\eta)/k} \rightarrow \Omega^i_{X^\circ_k/k}\otimes_{\O_{X^\circ}}\kappa(\eta)$ is injective. 
In particular, this implies that 
$h^{0,1\shortrightarrow i}_{X,D,p} > 0,$ and the second claim follows from Theorem \ref{thm:ellabessdimlowerbounds}.
\end{proof}

\begin{para} We say that $\alpha_X$ is unramified at $p$ if $\O_C$ and $X^\circ$ can be chosen 
so that the special fiber of $X^\circ$ is generically \'etale over its image under $\alpha_X^\circ.$ 
Recall that $X$ is said to have {\em maximal Albanese dimension} if its Albanese dimension is equal to $\dim X.$ 
Then we have the following.
\end{para}

\begin{cor}\label{cor:ellabessdimlowerbounds} Suppose $X$ is a smooth, proper $C$-scheme of maximal Albanese dimension.
If $p$ is a prime of unramified good reduction for $X$ at which $\alpha_X$ is unramified, and $Y\to X$ is the mod $p$ homology cover, then we have 
\[\ed(Y/X; p) = \dim X.\]
In particular, this holds for $p \gg 0.$
\end{cor}
\begin{proof} This follows from Proposition \ref{prop:albint}. \end{proof}

\section{Torus embeddings and Shimura varieties}

In this section, we use the results above to obtain lower bounds on the $p$-essential dimension of certain coverings that arise naturally in the context of torus embeddings and Shimura varieties. We assume that $p > 2$ throughout this section, so that $\F_p^\times$ is nontrivial.\footnote{This will allow us to deduce results from an analogue of ``weights'' with $\F_p^\times$ playing the role of $\mathbb{G}_m$.} In \S 3.1, we compute the cohomology of certain generalized finite Heisenberg groups. In \S 3.2 we consider certain torus bundles over abelian varieties. These have covers 
whose groups are the finite Heisenberg groups, and we use the results of \S 3.1 and Proposition \ref{prop:essdimlowerbounds} 
to show that these covers are $p$-incompressible. We then show a variant of this $p$-incompressibility result, 
where we restrict covers to analytic (not just Zariski) neighborhoods of the boundary in these torus bundles. Finally, in \S 3.3, 
we apply this last result to show $p$-incompressibility of certain congruence covers.

\subsection{Cohomology of generalized Heisenberg groups}\label{subsec:cohcenextns}
\begin{para} In this subsection, we will be concerned with central extensions of abelian groups, which are either finitely generated
free abelian pro-$p$-groups , or elementary abelian $p$-groups. We begin by considering a central extension of elementary abelian $p$-groups  
$$ 0 \rightarrow N \rightarrow E \rightarrow H \rightarrow 0.$$
Such extensions are classified by the cohomology group \cite[\S 10]{BensonCarlson}
$$ H^2(H,N) \simeq H^2(H, \F_p)\otimes N \simeq \wedge^2 H^{*}\otimes N \oplus H^{*}\otimes N . $$
Here the term $\wedge^2 H^{*}$ is the image of $H^1(H,\F_p)^{\otimes 2} = H^{*\otimes 2}$ in $H^2(H,\F_p)$ under the cup 
product, while the $H^{*}$ in the second term is the image of the Bockstein map 
$H^1(H,\F_p) \rightarrow H^2(H,\F_p)$ (see also \cite[Ch.~II., Cor.~4.3]{AdemMilgram}). 

Recall that if $c \in Z^2(H,N)$ is a $2$-cocycle, the corresponding central extension $E$ is defined by taking the underlying set of $E$ to 
be $N\times H$ with the group law given by 
\[(n_1,h_1)\cdot(n_2,h_2) = (n_1+n_2 + c(h_1,h_2), h_1+h_2).\] 
We call the extension $E$ a {\em finite Heisenberg group} \footnote{We remark that this is a small abuse of terminology, as this name is often reserved for the case when $\dim_{\F_p} N = 1.$} (or just {\em Heisenberg group} if the context is clear) if its class in $H^2(H,N)$ is represented by an alternating bilinear form 
$ H \otimes H \rightarrow N.$ 

For an $\F_p$-vector space $V$ with an action of $\F_p^\times,$ we say that {\em $V$ has weight 
$n \in \Z/(p-1)\Z$} if $\alpha \in \F_p^\times$ acts by $\alpha^{-n}.$ In general, we denote by $V_n \subset V$ the direct summand of weight $n.$
The following lemma gives a number of characterizations of Heisenberg groups, as well as showing that this property depends only on $E$ 
as a group, and not as an extension.
\end{para}

\begin{lem}\label{lem:Heisenbergcondns} The following conditions on the extension $E$ are equivalent: 
\begin{enumerate}[label=(\roman*)]
\item $E$ is a finite Heisenberg group.
\item The class $c \in H^2(H,N)$ defining $E$ is contained in $\wedge^2 H^{*}\otimes N .$
\item $E$ has exponent $p.$
\item The extension $E$ admits an action of $\F_p^\times,$ such that $H$ has weight $-1$ and $N$ has weight $-2.$
\end{enumerate}
\end{lem}
\begin{proof} {\em (ii) $\Rightarrow$ (i)}: Suppose $c$ satisfies {\em (ii)}. Write $c = \sum_{i=1}^k (\alpha_i \cup \beta_i )\otimes n_i\in Z^2(H,N),$ 
where $\alpha_i,\beta_i \in Z^1(H,\F_p)$ and $n_i \in N_i.$ Then $c$ is also represented by 
$$ \frac{1}{2} \sum_{i=1}^k (\alpha_i \cup \beta_i -  \beta_i \cup\alpha_i) \otimes n_i,$$
which is an $N$-valued alternating form.

{\em (i) $\Rightarrow$ (iv)}: Recall that, as a set, $E$ is identified with $N\times H.$ Define the $\F_p^\times$ action on $E$ by 
$\alpha\cdot (n,h) = (\alpha^2 n, \alpha h).$ This evidently induces an action of $\F_p^\times$ on $E$ as a set, and 
from the explicit description of the group law on $E$ one sees that this action respects the group structure 
(cf.~\cite[Prop.~6]{PeyreUnramified}).

{\em (iv) $\Rightarrow$ (iii)}: Consider the multiplication by $p$ map on $E.$ This gives a map $H \rightarrow N$ which commutes with any 
automorphism of $E$. In particular it commutes with the $\F_p^\times$-action, but $H$ and $N$ have distinct weights, so this map is $0.$

{\em (iii) $\Rightarrow$ (ii)}: First note that {\em (i) $\Rightarrow$ (iii)}, as, if $c$ is alternating, then $c(nh,h) = 0$ for any positive integer $n.$  
Now suppose $E$ has exponent $p,$ and let $c \in H^2(H,N)$ be its class.
Write $c = a + b$ in {\em (ii)} with $a \in \wedge^2 H^{*}\otimes N,$ and $b \in  H^{*}\otimes N .$ By what we just saw, 
the extension $E_{-a}$ corresponding to $-a$ has exponent $p.$ Replacing $E$ by the Baer sum of $E_{-a}$ and $E,$ we may 
assume that $c \in H^{*}\otimes N .$ If $c \neq 0,$ there exist linear forms $s:H^{*} \rightarrow \F_p$ and $t: N \rightarrow \F_p$ such that 
$s\otimes t (c) \in \F_p \simeq H^2(\F_p,\F_p)$ is nonzero. If $E$ has exponent $p$ then so does the extension $E_{s,t}$ obtained by 
pulling back and pushing out $E$ by $s$ and $t$ respectively. However $E_{s,t}$ corresponds to $s\otimes t (c)$ which is in the 
image of the Bockstein map, and hence  $E_{s,t} \simeq  \Z/p^2\Z,$ which is a contradiction. Hence $c=0,$ which implies {\em (ii)}. 
\end{proof}

\begin{para}\label{para:defnredn}
We now change our setup, and consider a central extension 
\begin{equation}\label{e:Heisext}
	0 \rightarrow N \rightarrow E \rightarrow H \rightarrow 0 
\end{equation}
where $N$ and $H$ are finitely-generated, free abelian pro-$p$ groups. 
That is $N \simeq  \Z_p^r$ and $H \simeq  \Z_p^s$ for some $r,s.$ 

For a finitely-generated free $\Z_p$-module $V,$ equipped with an action of $\Z_p^\times,$ we say that $V$ has weight 
$n \in \mathbb Z$ if $\alpha \in \Z_p^\times$ acts by $\alpha^{-n}$ on $V.$ Of course this notion makes sense for any 
character of $\Z_p^\times,$ but we will not need it in this generality.
\end{para}

\begin{lem}\label{lem:padicHeisenberg} For $E$ as in \eqref{e:Heisext} the following hold: 
\begin{enumerate}[label=(\roman*)] 
\item We have $H^2(H,N)\simeq \Lambda^2_{\Z_p}H^*\otimes N,$ so that 
the class of $E$ is represented by an alternating bilinear form.
\item There exists a $\Z_p^\times$ action on $E$ such that $H$ has weight $-1$ and $N$ has weight $-2.$
\end{enumerate}
\end{lem}
\begin{proof} The proof of {\em (i)} is analogous to the proof of {\em (ii) $\Rightarrow$ (i)} in Lemma \ref{lem:Heisenbergcondns}, 
and the proof of {\em (ii)} is then analogous to {\em (i) $\Rightarrow$ (iv)} in that lemma.
\end{proof}

\begin{lem}\label{lem:degenpadic} The Hochschild-Serre spectral sequence 
%
$$ E_2^{i,j} = H^i(H, H^j(N, \F_p)) \Rightarrow H^{i+j}(E,\F_p) $$
degenerates at the $E_3$-page 
if $p \geq \min\{\rk_{\Z_p} N+1, \rk_{\Z_p} H\}.$
\end{lem}
\begin{proof} The $\Z_p^\times$-action on $E$ constructed in Lemma \ref{lem:padicHeisenberg} acts compatibly on the spectral sequence, 
and induces an action of $\F_p^\times$ on the terms. 
The weights of this  $\F_p^\times$-action are elements of $\Z/(p-1)\Z.$ 

The source and target of $d_r:E^{i,j}_r \rightarrow E^{i+r,j-r+1}_r$ have weights $i+2j$ and $i+2j +(2-r)$ respectively. Thus, if $r = 3, \dots, p,$  
these weights are not equal and $d_r = 0.$ If $r \geq p+1,$ then our assumption on $p$ forces $d_r$ to be $0,$ 
as $E^{i,j}_r = 0$ unless $0 \leq i \leq \rk_{\Z_p} H,$ $0 \leq j \leq \rk_{\Z_p} N.$
\end{proof}

\begin{para} \label{para:modp}
A  {\em reduction mod $p$} of the extension $E$ of \eqref{e:Heisext} is a map of central extensions 
$$\xymatrix{
0  \ar[r] & N \ar[r]\ar[d] & E \ar[r]\ar[d] & H \ar[r]\ar[d] & 0 \\
0  \ar[r] & \bar N \ar[r]& \bar E \ar[r] & \bar H \ar[r] & 0
}
$$
which identifies $\bar N$ and $\bar H$ with $N/pN$ and $H/pH$ respectively, and such that $\bar E$ is a finite Heisenberg group. 

Below we will repeatedly use that for a finitely-generated, free abelian pro-$p$ group $F,$ 
and a finitely-generated elementary abelian group with trivial $F$-action $M,$ 
\[H^\bullet(F,M)=(\wedge^\bullet H^1(F,\F_p))\otimes M.\]
This is elementary (see e.g. \cite[Ch. 3.2, Example 1 and Ch. 3.3, Prop.14]{Serre}) and follows from K\"unneth for $M=\F_p$, and the universal coefficient theorem for general $M$.
\end{para}

\begin{lem}\label{lem:existredn} Any central extension $E,$ as in \eqref{e:Heisext} admits a mod $p$ 
reduction $\bar E.$ The extension class of $\bar E$ is uniquely determined by that of $E.$
\end{lem}
\begin{proof} Let $\bar N = N/pN,$ $\bar H = H/pH.$ The map 
$H^1(\bar H, \bar N) \rightarrow H^1(H, \bar N)$ is a bijection. 
Hence the composite 
$$ \wedge^2 H^1(\bar H, \F_p)\otimes \bar N \rightarrow H^2(\bar H, \bar N) \rightarrow H^2(H, \bar N) \simeq \wedge^2 H^1(H, \F_p)\otimes 
\bar N .$$ 
is again a bijection. This implies that the pushout of $E$ by $N \rightarrow \bar N$ arises from an extension $\bar E,$ 
whose class lies in the image of the first map above, and this class is uniquely determined by the class of $E.$
\end{proof}

\begin{prop}\label{cohredn} Let $E$ be as  in \eqref{e:Heisext}, and let $\bar E$ be a mod $p$ reduction of $E.$ 
If $p > \frac{1}{2}(\rk_{\Z_p} H +  3 \,\rk_{\Z_p} N+1)$, then the map of cohomology rings $H^\bullet(\bar E, \F_p) \rightarrow H^\bullet(E, \F_p)$ is surjective.
\end{prop}
\begin{proof} Recall the Hochschild-Serre spectral sequence, which converges to the mod $p$ cohomology of $E:$ 
$$ E^{i,j}_2  = \wedge^i \bar H^* \otimes \wedge^j \bar N^* = H^i(H,H^j(N, \F_p)) \Rightarrow H^{i+j}(E, \F_p).$$
By Lemma \ref{lem:degenpadic} and our assumptions on $p,$ the spectral sequence $E_r$ 
degenerates on the $E_3$ page.
There is also the analogous spectral sequence $\bar E_r$ converging to  $ H^{i+j}(\bar E, \F_p)$: 
$$ \bar E^{i,j}_2  = H^i(\bar H, \F_p)\otimes_{\F_p} H^j(\bar N, \F_p) = H^i(\bar H,H^j(\bar N, \F_p)) \Rightarrow H^{i+j}(\bar E, \F_p).$$

 As observed above (by \cite[\S 10]{BensonCarlson} or \cite[Ch.~II, Cor.~4.3]{AdemMilgram}), the cohomology ring $H^\bullet(\bar H, \F_p)$ is a tensor product of an exterior and a symmetric algebra: 
\[H^\bullet(\bar H, \F_p) = \wedge^\bullet \bar H^*(1) \otimes \Sym^\bullet \bar H^*(2),\] where 
$\bar H^*(1), \bar H^*(2),$ denote the $\F_p$-dual of $\bar H$ considered in degree $1$ and $2$ respectively. 
One also has 
the analogous description of $H^\bullet(\bar N, \F_p).$
Let $\tilde E^{i,j}_2 \subset \bar E^{i,j}_2$ be the subgroup generated by the image of 
$$ \wedge^i \bar E_2^{1,0} \otimes \wedge^j \bar E_2^{0,1} \simeq \wedge^i H^1(\bar H,\F_p)\otimes \wedge^j H^1(\bar N, \F_p) \simeq \wedge^i H^1(H,\F_p)\otimes \wedge^j H^1(N, \F_p) $$
under the cup product. The inclusion $\tilde E^{i,j}_2 \subset \bar E^{i,j}_2$ induces a splitting of the projection 
$\bar E^{\bullet,\bullet}_2 \rightarrow  E^{\bullet,\bullet}_2$ as graded groups. 

We now make use of the $\F_p^\times$-action on $\bar E$ given by Lemma \ref{lem:Heisenbergcondns}. 
This induces an $\F_p^\times$-action on the spectral sequence $\bar E_r.$ 
In particular, the differential 
$$ d_2: \bar N^* \simeq \bar E^{0,1}_2 \rightarrow \bar E^{2,0}_2 \simeq \wedge^2 \bar H^* \oplus \bar H^*$$
respects weights, and so satisfies $d_2(\bar E^{0,1}) \subset \tilde E^{2,0}_2.$
Since $\bar{E}^{\bullet,\bullet}_2$ is a differential graded algebra, 
this implies that $\tilde E^{\bullet,\bullet}_2$ is stable under $d_2.$  Denote the cohomology 
of $(\tilde E^{\bullet,\bullet}_2, d_2)$ by  $\tilde E^{\bullet,\bullet}_3.$ Then 
$\tilde E^{\bullet,\bullet}_3 \subset \bar E^{\bullet,\bullet}_3$ induces a splitting of the projection 
$\bar E^{\bullet,\bullet}_3 \rightarrow E^{\bullet,\bullet}_3,$ as graded groups, so again the latter map is surjective.

To prove the proposition it suffices to show that for $r \geq 2,$ and $0 \leq i \leq \rk_{\Z_p} H,$ $0 \leq j \leq \rk_{\Z_p} N,$ 
the map $\bar E^{i,j}_r \rightarrow E^{i,j}_r$ is surjective, as the target of this map is trivial for $i,j$ outside this range. 
Thus we assume from now on that $i,j$ satisfy these inequalities. 

The above description of the cohomology rings 
$H^\bullet(\bar H, \F_p),$ $H^\bullet(\bar N, \F_p)$ shows that the weights of $\bar E^{i,j}_2$ are represented by integers in the 
interval $[\frac{1}{2}(i+2j), i+2j].$ Our assumptions on $i,j$ and $p$ imply that this interval has length less than $p-1$, 
i.e.~ the set of integers in it maps injectively to $\Z/(p-1)\Z.$ 
In particular, the weight $i+2j$ piece $(\bar E^{i,j}_2)_{i+2j} \subset \bar E^{i,j}_2$ is precisely  $\tilde E^{i,j}_2,$
and so  $\tilde E^{i,j}_3 = (E^{i,j}_3)_{i+2j}.$ 

Now for $r \geq 3,$ consider the differential $d_r: \bar E^{i,j}_r \rightarrow \bar E^{i+r,j-r+1}_r.$ 
The weights appearing in $\bar E^{i+r,j-r+1}_r$ are represented by integers in 
$[\frac{1}{2}(i+2j-r+2), i+2j-r+2].$ Our assumptions on $i$, $j$, and $p$ imply that the interval $[\frac{1}{2}(i+2j-r+2),i+2j]$ has length less than $p-1$ for $r = 3, \dots, j+1.$  This implies that $(\bar E^{i+r,j-r+1}_r)_{i+2j} = 0,$ 
and so $d_r|_{(\bar E^{i,j}_r)_{i+2j} }= 0.$ On the other hand for $r > j+1$ $d_r|_{\bar E^{i,j}_r} = 0$ as the target of this map is trivial.

We now show by induction on $r$ that $(\bar E^{i,j}_r)_{i+2j}$ maps isomorphically to $E^{i,j}_r,$ which in particular implies that 
$\bar E^{i,j}_r \rightarrow E^{i,j}_r$ is surjective, finishing the proof of the proposition.
We have already seen this for $r=3.$ Assume the result for some $r \geq 3.$ As $d_r|_{(\bar E^{i,j}_r)_{i+2j} }= 0$ it follows that 
$$ (\bar E^{i,j}_r)_{i+2j}  \twoheadrightarrow H(\bar E_r)^{i,j}_{i+2j} \simeq (\bar E^{i,j}_{r+1})_{i+2j} \rightarrow  E^{i,j}_{r+1} \simeq E^{i,j}_r, $$
 where the final isomorphism follows from Lemma \ref{lem:degenpadic}. 
 The composite is the natural map, which we are assuming is an isomorphism. Hence all the maps above are isomorphisms, which 
 completes the induction. 
\end{proof}

\subsection{Toric varieties}\label{subsec:toric}
\begin{para} Throughout this section we work over an algebraically closed field $C$ of characteristic 0.  Let $T$ be a torus over $C,$ $S$ an abelian variety over $C,$ and $U/S$ a $T$-torsor. 
Fix  a geometric point $\bar x$ of $U,$ and let $\bar s$ be its image in $S.$   
We will apply the results of the previous section to $U.$  For a profinite group $G,$ denote by $G^p$ its maximal pro-$p$ quotient. 
\end{para}

\begin{lem}\label{lem:torusembedding} The $C$-scheme $U$ is an \'etale $K(\pi,1),$ and $\pi_1(U,\bar x)$ is a central extension of free abelian profinite groups
\begin{equation}\label{eqn:extn} 
0 \rightarrow \pi_{1,\et}(U_{\bar s}, \bar x) \rightarrow \pi_{1,\et}(U, \bar x) \rightarrow \pi_{1,\et}(S, \bar s) \rightarrow 0, 
\end{equation}
whose maximal pro-$p$ quotient is an extension 
\begin{equation}\label{eqn:propextn} 
0 \rightarrow \pi_{1,\et}(U_{\bar s}, \bar x)^p \rightarrow \pi_{1,\et}(U, \bar x)^p \rightarrow \pi_{1,\et}(S, \bar s)^p \rightarrow 0. 
\end{equation}
The natural map $H^\bullet(\pi_{1,\et}(U, \bar x)^p, \F_p) \rightarrow H^\bullet(\pi_{1,\et}(U, \bar x), \F_p)$ is an isomorphism. 
\end{lem}
\begin{proof} Note that $S$ and $U_{\bar s}$ are \'etale $K(\pi,1)$'s. Indeed, for each of these varieties, both their cohomology and the cohomology of the respective \'etale fundamental group are exterior algebras on their cohomology in degree $1.$ Now the fact that $U$ is a $K(\pi,1)$ follows by 
comparing the Hochschild-Serre spectral sequence for $\pi_{1,\et}(U_{\bar s}, \bar x) \subset \pi_{1,\et}(U, \bar x)$ and the Leray spectral sequence for 
the map $U \rightarrow S.$

It is well known that $\pi_{1,\et}(U,\bar x)$ is an extension as in \ref{eqn:extn}. That this is a central extension follows from the fact that the $T$-torsor $U$ is Zariski locally trivial (Hilbert's Theorem 90) \cite[IX, Thm.~3.3]{SGA4}. 

That the maximal pro-$p$ quotient of (\ref{eqn:extn}) is an extension as in (\ref{eqn:propextn}) follows easily from the fact that any central extension of two finite abelian groups of coprime order is trivial. It follows that the kernel 
of $\pi_{1,\et}(U, \bar x) \rightarrow \pi_{1,\et}(U, \bar x)^p$ is a prime to $p$ profinite group, which implies the final claim.
\end{proof}

\begin{para}\label{para:Heisenbergquot} 
Keep the assumptions above, and consider the extension (\ref{eqn:propextn}). 
We have the notion of a mod $p$ reduction of such an extension given in \ref{para:modp}. 
In particular, such a mod $p$ reduction gives rise to a 
surjection of $\pi_{1,\et}(U,\bar x) \twoheadrightarrow \bar E$ onto a finite Heisenberg group $\bar E.$ 
We write $U(\bar E) \rightarrow U$ for the finite cover corresponding to $\bar E.$ 
\end{para}

\begin{para}
We again denote by $T$ the split torus over $\mathbb Z$ with character group $X^*(T).$ 
We say $U \rightarrow S$ has {\em good reduction at $p$} if there exists a $p$-adic valuation on $C,$ with ring of integers $\O_{C},$ such that  $U \rightarrow S$ extends to a $T$-torsor over an abelian scheme over $\O_{C^\circ},$ $U^\circ \rightarrow S^\circ.$ 
 We say $U \rightarrow S$ has {\em unramified good reduction} if $U^\circ \rightarrow S^\circ$ can be chosen so that it descends to a $T$-torsor 
 over an absolutely unramified discrete valuation ring $\O \subset \O_{C}.$
\end{para}

\begin{prop}\label{prop:edTtorsor} Suppose that $p > \dim U + \frac{1}{2}(\dim U_{\bar s} + 1),$ and that $U\rightarrow S$ has unramified good reduction at $p.$ 
Then $ \ed(U(\bar E)/U;p) = \dim U.$
\end{prop}
\begin{proof} Since $U$ has unramified good reduction, there is a $p$-adic valuation on $C,$ and an 
absolutely unramified discrete valuation ring $\O \subset \O_{C},$ such that $U$ descends to a $T$-torsor over an abelian scheme 
$U^\circ \rightarrow S^\circ$ over $\O.$ Fix a basis for $X^*(T).$ Then $U^\circ$ corresponds to a collection of line bundles $\L_1, \dots, \L_t$ over $S^{\circ},$ where $t = \dim U_{\bar s}.$ Let $P^\circ_i = \Proj_{\O_{S^{\circ}}}(\O_{S^{\circ}}\oplus \L_i^\vee),$ and set 
$P^\circ = P^\circ_1\times_{S^{\circ}} P^\circ_2 \times_{S^\circ} \dots \times_{S^{\circ}} P^\circ_t $.   
Then $P^\circ$ is a smooth projective scheme $S^\circ$-scheme, and $D^\circ = P^\circ - U^\circ$ is a normal crossings divisor. 
We set $P = P^\circ_{C}$ and $D = D^\circ_{C}.$

Let $d = \dim U.$ Fix an identification $T \simeq  \GG^t_m,$ and let  $z_1, \dots z_t$ be the 
standard co-ordinates on $\GG^t_m.$ 
By Hilbert's Theorem 90, Zariski locally on $S,$ we can identify $U$ 
with $S\times T \simeq S \times \GG^t_m.$ The differential 
$\omega_T = \frac{dz_1}{z_1} \wedge \dots \wedge\frac{dz_t}{z_t}$ is $T$-invariant, and hence does not depend on the identification 
$S\times T \simeq U.$ (It  is independent of our fixed isomorphism $T \simeq  \GG^t_m$ up to a sign). It follows that $\omega_T$ 
gives rise to a global section of $\Gamma(P,\Omega^t_P(\log D)).$ Let $r = \dim S,$ and $\omega_r \in \Gamma(S, \Omega^r_S)$ a nonzero $r$-form. 
Then $\omega_r\otimes \omega_T \in \Gamma(P,\Omega^d_P(\log D))$ is nonzero, so $h^{0,d}_{(P,D)} \neq 0.$

Now consider the maps 
$$ H^d(\bar E, \F_p) \rightarrow H^d(\pi_1(U, \bar x)^p, \F_p) \rightarrow H^d(\pi_1(U, \bar x), \F_p) \rightarrow H^d(U, \F_p).$$ 
The first map is surjective by Proposition \ref{cohredn}, and the second and third maps are isomorphisms by Lemma \ref{lem:torusembedding}. 
As $h^{0,d}_{(P,D)} \neq 0,$ the Proposition follows from Proposition \ref{prop:essdimlowerbounds}.
\end{proof}

\begin{para}\label{para:toric} We want to prove a variant of Proposition \ref{prop:edTtorsor} which will be used in the next subsection to show $p$-incompressibility of certain coverings of  Shimura varieties. 
To explain it, we recall some facts about families of toric varieties \cite{Fultontoric}, \cite{Kempfetal}, 
\cite[\S 5]{Pinkthesis}. Note that the family $P \rightarrow S,$ which appeared in the proof of Proposition \ref{prop:edTtorsor} is an example of such a family.

Let $k$ be a field. A toric variety, or torus embedding $X$ over $k$ is a normal $k$-scheme $X,$ equipped with an action of a split torus $T,$ such that 
$X$ admits a dense open $T$-orbit $U,$ which is a $T$-torsor. These can be described in terms of 
{\em fans}, which are certain collections $\Sigma$ of {\em convex polyhedral cones} in $X_*(T)_{\mathbb \Q}$.
In particular, proper toric varieties correspond to {\em complete} fans, which are certain decompositions of 
$X_*(T)_{\mathbb \Q}$ into convex polyhedral cones. 

Since we are assuming that $T$ is a split torus, it extends canonically to a torus over $\mathbb Z$, which we again denote by $T.$  
Thus, a fan $\Sigma$ actually defines a torus embedding $X(\Sigma)$ over $\mathbb Z.$ That is, $X(\Sigma)$ is a normal scheme, 
equipped with an action of $T,$ containing a dense $T$-torsor. The fiber of $X(\Sigma)$ over any point $ \Spec k \to\Spec \mathbb Z$ is the toric variety over $k$ corresponding to $\Sigma.$

We will need the relative version of this notion. Let $S$ be a scheme and $X \rightarrow S$ a map of schemes, 
equipped with an action of $T.$ Then $X/S$ is called a torus embedding over $S$ if, Zariski locally on $S,$  $X \rightarrow S$ can 
be $T$-equivariantly identified with $X(\Sigma)\times_{\mathbb Z} S$ for a fan $\Sigma.$ 
In this case there is an open subset $U \subset X,$ which is a $T$-torsor, and which is dense in the fiber over every point of $S.$ 
The complement $D = X-U,$ is called the {\em boundary} of $X.$ 

We remark that it may appear more natural to make this definition with the condition on $X\rightarrow S$ imposed only \'etale locally on $S.$ However these two notions are the same: As $X(\Sigma)$ contains an open dense $T$-torsor $U(\Sigma),$ 
 $\Aut_T (X(\Sigma)\times_{\mathbb Z} S) = \Aut_T (U(\Sigma)\times S) = T(S).$ Hence any $X\rightarrow S$ which is \'etale locally isomorphic to 
 $X(\Sigma)\times S$ gives rise to an \'etale $T$-torsor. As in the proof of Proposition~\ref{prop:edTtorsor}, such a torsor is Zariski locally trivial.  
\end{para}

\begin{lem}\label{classificationtoric} Suppose that $S$ is irreducible, and fix $T$ as above. Then there is an equivalence of 
categories between torus embeddings $X \rightarrow S$ containing a dense open $T$-torsor, and pairs $(U,\Sigma),$ where $U$ is a $T$-torsor over $S,$ and $\Sigma$ is a fan in $X_*(T)_{\Q}.$ 
\end{lem}
\begin{proof} Given $X \rightarrow S,$ we can associate to it the open $T$-torsor $U \subset X,$ and the fan $\Sigma$ corresponding to the 
torus embedding $X_{\eta},$ where $\eta \in S$ is the generic point. Conversely, given $(U,\Sigma),$ $\Sigma$ defines a toric variety 
$X_0$ over $C,$ which is even defined over $\Z.$ We take $X = (X_0 \times U)/T.$ One checks easily that these two constructions 
are quasi-inverses.
\end{proof}

\begin{para} Now suppose that $S$ is a $C$-scheme, and let $X \rightarrow S,$ $U$ and $D$ be as above. 
Then $D$ has a stratification $D_0 \subset D_1 \subset D_2 \dots,$ which may be described as follows: 
Zariski locally, $X = X_0 \times S$ as torus embeddings, where $X_0$ is a torus embedding over $C.$ We set $D_j \subset D$ to be the product of $S$ and the closure of the $j$-dimensional $T$-orbits in $X_0.$ Then $D_0$ is \'etale over $S,$ and, if $X$ is proper, it is finite \'etale.

Let $X^1 \rightarrow X$ be the blow up of $D_0$ on $X.$ Since $D_0$ is fixed by $T,$ $X^1$ is again a torus embedding over $S,$ 
and we denote its boundary by $D^1.$  If $X$ is smooth over $S,$ with $D$ a relative normal crossings divisor, then the same is true for $X^1$ 
and $D^1.$ We can continue this construction to obtain a sequence of blow ups $X \leftarrow X^1 \leftarrow X^2 \dots,$ with boundary 
$D^i \subset X^i.$ Finally, we denote by $\widehat X^i$ the completion of $X^i$ along $D^i$, and by $\widehat X$ the completion of $X$ along $D$. 

The following lemma implies that if $(X,D)$ has unramified good reduction at $p,$ then so does $(X^1, D^1):$ 
\end{para}

\begin{lem}\label{toricgoodredn} Suppose that $X/S$ is smooth, and that $D$ is a normal crossings divisor. 
If the $T$-torsor $U$ has unramified good reduction at $p$, then $(X,D)$ has unramified good reduction at $p.$
\end{lem}
\begin{proof} Suppose $U/S$ descends to a $T$-torsor $U^\circ$ over a discrete valuation ring $\O \subset C.$ 
By Lemma \ref{classificationtoric}, as a torus embedding $X$ is determined by $(U,\Sigma)$ for some fan $\Sigma.$ 
Then $(U^\circ, \Sigma)$ determines a torus embedding $X^\circ$ over $\O.$ The property that $X$ is smooth and $D$ is a 
normal crossings divisor can be read off from the fan $\Sigma$ \cite[II.2,Thm.~4*]{Kempfetal}. 
It implies that $X^\circ$ is smooth over $\O,$ and that its boundary $D^\circ$ is a relative normal crossings divisor.
\end{proof} 

\begin{lem}\label{lem:blowups} Suppose that $S$ and $X/S$ are smooth, with $D \subset X$ a relative normal crossings divisor. 
Let $Z \subset \widehat X$ be a formal subscheme of codimension $1,$ and $Z^i \subset \widehat X^i$ its proper transform. 
Then there exists a dense open subset $W \subset S$ and an integer $i \geq 1$ such that $D_0^i|_W \nsubseteq Z^i.$
\end{lem}
\begin{proof} Let $\pi_i: \widehat X^i \rightarrow \widehat X.$ Recall that the proper transform $Z^i$ is defined to be the union of those 
components of $\pi_i^{-1}(Z)$ which are not contained in $\pi_i^{-1}(D_0).$ It suffices to show the lemma with $S$ replaced by one of its 
generic points, so we may assume that $S = \Spec \kappa$ is a field. 
By induction on the number of irreducible components of $Z,$ we may assume that $Z$ is an irreducible Cartier divisor in $X.$ 

Let $x \in D_0,$ and let $z_1, \dots z_n$ be a system of local co-ordinates at $x,$ such that the ideal of $D$ at $x$ is given by 
$z_1z_2\dots z_n.$ Then in a formal neighborhood of $x$, the ideal of $Z$ is generated by a nonzero power series $f = \sum_I a_I z^I,$ 
where $I = (i_1, \dots, i_n)$ runs over $n$-tuples of non-negative integers, and $a_I \in \kappa.$ Choose such an $n$-tuple 
$J = (j_1, \dots, j_n)$ such that $a_J \neq 0$ and $|J| = j_1 + \dots + j_n$ is as small as possible. If $|J| = 0,$ then $x \notin Z,$ 
and the lemma holds without blowing up $X.$ In general we proceed by induction on $|J|.$

We may assume without loss of generality that $j_1 \neq 0.$ There is a point $x^1 \in D^1_0 \subset \widehat X^1$ such that 
the functions $z_1, u_2, u_3, \dots , u_n,$ with $u_i = \frac{z_i}{z_1}$ are a system of local co-ordinates at $x^1,$ and the 
ideal of $D^1$ is generated by $z_1u_2\dots u_n.$ The ideal of the proper transform $Z^1 \subset \widehat X^1$ is generated by the function 
$z_1^{-|J|} f,$ whose expansion contains the term $a_I u_2^{j_2}\dots u_n^{j_n}.$ Thus the result follows by induction on $|J|.$
\end{proof}

\begin{para} We now assume that $C = \mathbb C,$ the complex numbers, and we denote by $X^{\an}$ the complex analytic space associated to $X.$ 
Recall that a {\em Zariski closed} subset of a complex analytic space $Y$ is a closed subset $Z \subset Y,$ which is locally defined by 
(local) analytic functions on $Y.$ We call the complement of a Zariski closed analytic subset a {\em Zariski open} subset of $Y.$  Note that if $Y\to W$ is an open embedding of complex analytic spaces, it is not in  general the case that a Zariski open $U\subset Y$ is a Zariski open of $W$; this is only the case if the closed analytic subset $Z\subset Y$ is of the form $Z=Y\cap Z'$ for some closed analytic subset $Z'\subset W$.
\end{para}

\begin{prop}\label{prop:tubular} Let $S$ be an abelian scheme and let $X/S$ be a proper torus embedding with boundary $D \subset X,$ 
and dense open $T$-torsor $U \subset X.$  Let $d = \dim X,$ and suppose that $U/S$ has unramified good reduction at a prime $p> d+2.$ 

Let $V \subset X^{\an}$ be an analytic open subset containing $D^{\an},$ and $V' \subset V$ a non-empty Zariski open subset of $V$. Then the map $$H^d(U,\F_p) \rightarrow H^d(U\cap V', \F_p)$$ is nonzero.
\end{prop}
\begin{proof} After replacing $X$ by a blow up, we may assume that $X/S$ is smooth and $D \subset$ $X$ is a normal crossings divisor. 
Such blow ups are obtained by subdividing the rational  cone decomposition of $X_*(T)_{\Q}$ defining $X,$ see \cite[p 48]{Fultontoric}.  
Then $(X,D)$ has unramified good reduction at $p$ by Lemma \ref{toricgoodredn}. 

Next the same proof as in Proposition \ref{prop:edTtorsor} shows that $h^{0,d}_{(X,D)} \neq 0.$ In fact the differential $\omega_r\otimes \omega_s$ 
defined there is in $H^0(X, \Omega^d_X(\log D))$ for any torus embedding $(X,D)$ with open dense $T$-torsor $U \subset X.$ 

Shrinking $V'$ as necessary, we may assume that $Z =  V - V'$ has codimension $1$ in $V.$ Let $\widehat Z \subset \widehat X$ denote the formal completion of 
$Z$ along $D.$ We apply Lemma \ref{lem:blowups} to $\widehat Z.$ Then, after replacing $X$ by a blow up and $Z$ by its proper transform, 
we may assume there 
is a non-empty Zariski open $W \subset S$ such that $D_0|_W \nsubseteq Z.$ Shrinking $W$ if necessary, since $D_0$ is \'etale over $S,$ 
we may assume that there is an irreducible component $D'_0$ of $D_0$ which does not meet $Z,$ and that $X|_W = W \times X_0$ 
for a toric variety $X_0$ over $\mathbb C.$ Then $D_0'$ has the form $W \times x_0$ for a $T$-fixed point $x_0 \in X_0.$ 
We write $U_0 \subset X_0$ for the open $T$-orbit.

By \cite[Thm.~4]{Loj}, there exists a compact subset $W^- \subset W$ such that the inclusion induces a homotopy equivalence. 
Denote by $D_{X_0}$ the boundary of $X_0.$ Let $Y_0 \subset X_0$ be a ball around $x_0$ which is small enough that 
$Y_0 \backslash D_{X_0} \rightarrow U_0$ is a homotopy equivalence, 
$W^- \times  Y_0 \subset V,$ and $W^- \times Y_0$ does not meet $Z.$ 
Here we are using the compactness of $W^-$ for the second and third properties.  
Let  
$$Y^* = (W^- \times Y_0)\backslash D \simeq W^-\times (Y_0\backslash D_{X_0}).$$ 
Then $Y^*   \rightarrow U_{W^-} \simeq W^-\times U_0$ is a homotopy equivalence, and $Y^* \subset U \cap V'.$   Now consider the composite 
$$ H^d(U, \F_p) \rightarrow H^d(U|_{W}, \F_p) \rightarrow H^d(U|_{W^-}, \F_p) \rightarrow H^d(Y^*, \F_p). $$
Since  $h^{0,d}_{X,D} \neq 0,$ and $(X,D)$ has unramified good reduction at $p,$ we may apply 
Corollary \ref{cor:goodredn} to deduce that the first map is nonzero. The other two maps are induced by homotopy equivalences, hence 
are isomorphisms. Thus the composite map is nonzero. However, as $Y^* \subset U \cap V',$ the composite map factors through 
$H^d(U\cap V', \F_p)$ which implies the Lemma.
\end{proof}

\begin{para}\label{analyticedp} We are nearly ready to show the variant of Proposition \ref{prop:edTtorsor} that will be used in the next subsection. 
To formulate it we need a notion of essential dimension at $p$ for complex analytic spaces. This is defined in a similar way as for algebraic varieties, but there is an important difference in that we do not insist the auxiliary coverings of order prime to $p$ are \'etale. 
Unlike the algebraic situation, one cannot reduce ramified coverings 
to the unramified case, because the notion of Zariski open subsets of complex analytic spaces is not transitive. 

Let $V_1 \rightarrow V$ be a finite map of (reduced) complex analytic spaces. The {\it essential dimension} $\ed(V_1/V)$ is the smallest integer $e$ 
such that for some dense Zariski open $V' \subset V, $ there exists a finite map of analytic spaces $Y_1 \rightarrow Y,$ and a 
map $V' \rightarrow Y$ 
such that $V_1|_{V'} \rightarrow V'$ is isomorphic to the normalization of $ V'\times_YY_1.$ The $p$-{\it essential dimension} $\ed(V_1/V_2; p )$ 
is defined as the minimum value of $\ed(\tilde V_1/\tilde V),$ where $\tilde V \rightarrow V$ runs over finite maps of degree prime to $p,$ 
and $\tilde V_1$ is the normalization of $\tilde V\times_VV_1.$ Note that the map $\tilde V \rightarrow V$ is finite flat over some Zariski open; 
its degree is defined as the degree over any such Zariski open.
\end{para}

\begin{para} Keep the notation of Proposition \ref{prop:tubular} , and consider a surjection onto a finite Heisenberg group 
$\pi_{1,\et}(U,\bar x) \rightarrow \bar E,$ as in \ref{para:Heisenbergquot}. We denote by $X(\bar E) \rightarrow X$ the normalization 
of $U(\bar E) \rightarrow U.$
\end{para}

\begin{cor}\label{cor:analyticedp} With the assumptions of Proposition \ref{prop:tubular}, suppose that 
$p > \dim U + \frac{1}{2} (\dim U_{\bar s} + 1),$ and write $V(\bar E):= X(\bar E)^{\an}|_{V}.$ 
Then $$\ed(V(\bar E)/V; p) = d.$$
\end{cor}
\begin{proof} Let $\pi: \tilde V \rightarrow V$ be a finite covering of degree prime to $p,$ and $\tilde V' \subset \tilde V $ a dense Zariski open. 
We claim that the map $H^d(U,\F_p) \rightarrow H^d(\tilde V'|_U, \F_p)$ is nonzero.

The complement $\tilde Z = \tilde V \backslash \tilde V',$ is a Zariski closed subspace of everywhere positive codimension. Hence 
$\pi(\tilde Z) \subset V$ is Zariski closed with everywhere positive codimension. Let $V' = V\backslash \pi(\tilde Z).$ 
We may replace $\tilde V'$ by the preimage of $V',$ and assume that $\pi$ restricts to $\tilde V' \rightarrow V'.$  Now consider the composite 
$$H^d(U,\F_p) \rightarrow H^d(U \cap V', \F_p) \rightarrow H^d(\tilde V'|_U, \F_p). $$
The first map is nonzero by Proposition \ref{prop:tubular}, and the second map is injective, as in the proof of Proposition \ref{prop:essdimlowerbounds}, 
hence the claim. As in the proof of Proposition \ref{prop:edTtorsor}, it follows that the composite 
$$ H^d(\bar E, \F_p) \rightarrow H^d(U,\F_p) \rightarrow  H^d(\tilde V'|_U, \F_p)$$ 
is nonzero.

Write $\tilde V'(\bar E)$ for the normalization of $V(\bar E)\times_{V}\tilde V'.$ If $\tilde V'(\bar E) \rightarrow \tilde V'$ is the normalized pullback 
of an $\bar E$-covering $Y_1 \rightarrow Y$ of dimension $< d,$ then as in the proof of Proposition \ref{prop:essdimlowerbounds}, 
the map $ H^d(\bar E, \F_p) \rightarrow H^d(\tilde V'|_U, \F_p)$ factors through 
$H^d(Y, \F_p).$ Shrinking $Y$ and $\tilde V'$ as necessary, we can assume that $Y$ is Stein. By Andreotti-Frankel \cite{AF}, $\dim Y<d$ implies that $H^d(Y,\F_p)=0.$ This gives a contradiction, and proves the corollary.
\end{proof}

\begin{cor}\label{cor:analyticedp2} Keep the assumptions of Corollary \ref{cor:analyticedp}, but suppose that $\dim S = 0.$
Then the conclusion of  \ref{cor:analyticedp} holds with no restriction on $p.$ 
\end{cor}
\begin{proof} The restriction on $p$ in the proof of \ref{cor:analyticedp} come from the application of Corollary \ref{cor:goodredn} 
and Proposition \ref{cohredn}, which guarantee that the maps 
$$ H^d(\bar E, \F_p) \rightarrow H^d(U, \F_p) \rightarrow H^d(U|_W, \F_p) $$ 
are injective for $W \subset S$ dense open. When $\dim S = 0,$ the second map is vacuously bijective, and the first map is surjective, 
because in this case $\pi_{1,\et}(U,\bar x)$ is abelian, $\bar E = \pi_{1,\et}(U,\bar x)/p\pi_{1,\et}(U,\bar x),$ and the cohomology ring 
$H^\bullet(U, \F_p)$ is generated in degree $1.$
\end{proof}

\begin{para}
We remark that Corollary \ref{cor:analyticedp2} is originally due to Burda \cite{Burda}, and our proof in this case 
reduces to a variant of his. The key point in the argument is the construction of the subset $Y^*$ in the proof of Propostion \ref{prop:tubular}. 
Burda does this by considering annuli in $U$ with carefully chosen radii, rather than by using blow ups; see \cite[Thm.~18]{Burda}.
\end{para}

\subsection{Shimura Varieties}

\begin{para} Let $(G,X)$ be a Shimura datum \cite[\S 1]{Deligne}. Recall that this consists of a reductive group $G$ over $\Q,$ together 
with a $G(\RR)$-conjugacy class, $X,$ of homomorphisms $h:  \SS =  \Res_{\CC/\RR} \GG_m \rightarrow G$ satisfying certain conditions.
These imply, in particular, that $X$ is a Hermitian domain, and that for any neat compact open $K \subset G(\AA_f),$ 
the quotient 
$$ \Sh_K(G,X) = G(\Q) \backslash X \times G(\AA_f)/K $$
has the structure of a complex algebraic variety. Here $\AA_f$ denote the finite adeles over $\Q.$

For $h \in X,$ define a $G$-valued cocharacter $\mu_h$ over $\CC,$ as follows. 
For a $\CC$-algebra $R,$ we have $R\otimes_{\RR} \CC = R \oplus c^*(R),$ where $c$ denotes complex conjugation. 
The first factor gives an inclusion $R^\times \subset (R\otimes_{\RR} \CC)^\times,$ which gives a map $\GG_m \rightarrow \SS$ 
over $\CC,$ and $\mu_h$ is the composite of this map and $h.$
The cocharacter $\mu_h$ is miniscule \cite[1.2.2]{Deligne}, and the axioms for a Shimura variety imply that 
it is nontrivial.
\end{para}

\begin{para}
Now consider the Dynkin diagram $\Delta(G)$ of $G,$ which is equipped with an action of $\Gal(\bar \Q/\Q).$ 
To simplify the discussion, we assume from now on that $G^{\ad}$ is $\Q$-simple, so that the $\Gal(\bar \Q/\Q)$-action 
on $\Delta(G)$ is {\em transitive}. The subgroup of $\Gal(\bar \Q/\Q)$ that acts on $\Delta(G)$ trivially corresponds to a field $K_{\Delta}$ that is either CM or a totally real field. The action of complex conjugation $c \in \Gal(K_{\Delta}/\Q)$ is given by the 
{\em opposition involution} of the root system of $G.$ 
Recall that this is given by $(-1)\circ w_0,$ where $w_0$ is the longest element of the Weyl group $W_G.$ 
Thus $K_D$ is a totally real field exactly when $-1 \in W_G.$ 

Fix a maximal torus $T \subset G^{\ad}$ and a set of positive roots $\Delta^+\subset X^*(T)$ for $G.$ 
The vertices of $\Delta(G)$ correspond to the simple roots in $\Delta^+.$ 
For $\alpha \in \Delta(G)$, let $\mu_{\alpha} \in X_*(T)$ be the cocharacter 
which takes the value $1$ on $\alpha$ and vanishes otherwise.
For a subset $R \subset \Delta(G)$, define $\mu_R \in X_*(T)$ by $\mu_R:= \prod_{\alpha \in R} \mu_{\alpha}.$
If we do not specify $\Delta^+,$ then the conjugacy class of $\mu_R$ is still well defined, and we will denote it  
by $[\mu_R].$

Let $P_R \subset G$ be the parabolic associated to $\mu_R.$ 
Recall that $P_R$ is characterized by the condition that its Lie algebra $\Lie P_R$ is the sum of root spaces on which $\mu_R$ is 
non-negative \cite[XXVI, 1.4]{SGA3}. As for $\mu_R,$ if we do not specify $\Delta^+$ then the conjugacy class of $P_R$ 
is still well defined, and we denote it by $[P_R].$  

Since $\mu_h$ is miniscule, it corresponds to a collection of vertices $\varSigma \subset \Delta(G),$ 
with each component of $\Delta(G)$ containing at most one element of $\varSigma.$ 
Let $\tilde \varSigma = \varSigma \cup c(\varSigma).$

\begin{lem}\label{lem:strucunipot} Let $\mu_{\tilde \varSigma}$ be in $[\mu_{\tilde \varSigma}],$ $P_{\tilde \varSigma}$ the associated 
parabolic and $U_{\tilde \varSigma} \subset P_{\tilde \varSigma}$ the unipotent radical of $P_{\tilde \varSigma}.$ 
Then $U_{\tilde \varSigma}$ is a central extension of additive groups
$$ 0 \rightarrow Z(U_{\tilde \varSigma}) \rightarrow U_{\tilde \varSigma} \rightarrow U_{\tilde \varSigma}/Z(U_{\tilde \varSigma})\rightarrow 0,$$ 
and $U_{\tilde \varSigma} = Z(U_{\tilde \varSigma})$ if $c$ fixes $\varSigma.$
The cocharacter $w_{\tilde\varSigma} = \mu_{\varSigma}\mu_{c(\varSigma)}$ acts with weight $2$ on 
$Z(U_{\tilde \varSigma})$ and weight $1$ on $U_{\tilde \varSigma}/Z(U_{\tilde \varSigma}).$
\end{lem}
\begin{proof} Write $G^{\ad}_{\RR} = \prod_{i=1}^r G^{\ad}_{\RR,i},$ where each factor $G^{\ad}_{\RR,i}$ is absolutely simple. 
For each $i$ the Dynkin diagram of $G^{\ad}_{\RR,i}$ corresponds to a component  $\Delta(G)_i$ of $\Delta(G),$ and it suffices to prove 
the statement of the Lemma with $U_{\tilde\varSigma,i} = U_{\tilde\varSigma} \cap G^{\ad}_{\CC,i}$ in place of $U_{\tilde\varSigma}.$ 
We set  $\varSigma_i =  \varSigma \cap \Delta(G)_i$ and $\tilde \varSigma_i = \tilde \varSigma \cap \Delta(G)_i.$

The root spaces which appear in $\Lie U_{\tilde \varSigma}$ are exactly those which correspond to roots whose expressions as a sum of simple roots in $\Delta^+$ contain an element of $\tilde\varSigma.$ 
If $\varSigma_i$ is empty, then $U_{\tilde\varSigma,i} = 0,$ and there is nothing to prove. Suppose that $\varSigma_i$ is non-empty.
 
If $c$ fixes $\varSigma_i$ then $\tilde \varSigma_i = \varSigma_i$ and $\mu_{\varSigma_i}$ is nonzero on exactly one simple root in $\Delta(G)_i.$ 
As $\mu_h$ is miniscule, so is $\mu_{\varSigma_i},$ so the root in $\varSigma_i$ can appear in the expression for a root in $\Delta^+$ with multiplicity at most $1.$ This implies that $U_{\tilde \varSigma,i}$ is abelian, and $\mu_{\varSigma_i}$ acts with weight $1$ on $U_{\tilde \varSigma,i},$ 
so $\mu_{\varSigma_i}\mu_{c(\varSigma_i)} = \mu_{\varSigma_i}^2$ acts with weight $2.$  

Now suppose that $c$ does not fix $\varSigma_i.$ Then $\tilde \varSigma_i$ has exactly two elements 
and $\mu_{\tilde \varSigma}$ is nonzero on the two corresponding simple roots, which are exchanged by $c.$ 
If $\alpha \in \varSigma_i,$ and $e_{\alpha}$ is the corresponding root, then $e_{\alpha}$ and $e_{c(\alpha)}$ both appear in the 
expression for the longest root in $\Delta^+$ as a sum of simple roots. Using this, one sees that the root spaces appearing in 
$\Lie Z(U_{\tilde \varSigma,i})$ 
correspond to $\beta \in \Delta^+$ whose expression as a sum of simple roots contains both $\alpha$ and $c(\alpha),$ 
and that $[U_{\tilde \varSigma_i}, U_{\tilde \varSigma_i}] \subset Z(U_{\tilde \varSigma_i}).$ 
Moreover  the description of the root spaces appearing in the Lie algebras of $U_{\tilde\varSigma,i}$ and $Z(U_{\tilde \varSigma,i})$ 
implies the claim about the weights of $\mu_{\varSigma_i}\mu_{c(\varSigma_i)} = \mu_{\tilde\varSigma_i}.$
\end{proof}

\begin{para}
We will be interested in the following condition
\begin{equation}\label{rationalparabolic}
\text{\it The conjugacy class $[\mu_{\tilde \varSigma}]$ contains a cocharacter defined over $\Q.$}
\end{equation}
\end{para}

Note that this condition implies that $[\mu_{\tilde \varSigma}]$ is fixed by $\Gal(\bar \Q/\Q).$ 
We can rephrase this condition in terms of the parabolic $P_R \subset G_{\CC}$ associated to $\mu_R.$ 
\end{para}

\begin{lem}\label{lem:equivrational} The conjugacy class $[\mu_{\tilde \varSigma}]$ contains a cocharacter defined over $\Q$ 
if and only if $[P_{\tilde \varSigma}]$ contains a parabolic defined over $\Q.$
\end{lem}
\begin{proof} If $[\mu_{\tilde\varSigma}]$ contains a cocharacter $\mu_{\tilde\varSigma,\Q}$ defined over $\Q,$ then 
the subspace of $\Lie G$ on which $\mu_{\tilde\varSigma,\Q}$ is non-negative is the Lie algebra of a parabolic 
in $[P_{\tilde\varSigma}],$ which is defined over $\Q.$ 

Conversely if $[P_{\tilde\varSigma}]$ contains a parabolic $P_{\tilde\varSigma,\Q}$ defined over $\Q,$ then 
$P_{\tilde\varSigma, \Q}$ is associated (as above) to some cocharacter $\mu_{\tilde \varSigma}$ (not necessarily defined over $\Q$) 
in $[\mu_{\tilde \varSigma}].$ Let $U_{\tilde\varSigma, \Q} \subset P_{\tilde\varSigma, \Q}$ denote the unipotent radical. 
As $P_{\tilde\varSigma, \Q}$ is its own normalizer, $\mu_{\tilde\varSigma}$ is determined up to conjugation by points of $P_{\tilde\varSigma, \Q}.$ 
Hence the conjugacy class of $\mu_{\tilde \varSigma}$ as a $P_{\tilde\varSigma, \Q}/Z_G$-valued cocharacter is defined over $\Q,$ 
where $Z_G$ denotes the center of $G.$ On the other hand, the composite 
$$ \GG_m \overset {\mu_{\tilde \varSigma}} \rightarrow P_{\tilde\varSigma, \Q}/Z_G \rightarrow P_{\tilde\varSigma, \Q}/U_{\tilde\varSigma, \Q}Z_G$$
is central, hence defined over $\Q.$ Now this composite can be lifted to a $P_{\tilde\varSigma, \Q}/Z_G$-valued cocharacter defined over $\Q,$ 
and any such lift is in $[\mu_{\tilde\varSigma}]$, by \cite[IX, Thm.~3.6]{SGA3}.
\end{proof}

\begin{para} Somewhat more explicit conditions which guarantee that \ref{rationalparabolic} holds are given by the following lemma.
\end{para}

\begin{lem}\label{lem:rationalcocharacter} Suppose that $G^{\ad}$ is $\Q$-simple. Then $G^{\ad} = \Res_{F/\Q} G_0$ where $F$ is a totally real field, and $G_0$ is an absolutely simple group over $F.$ 
If 
\begin{enumerate}
\item $\tilde \varSigma$ is $\Gal(\bar \Q/\Q)$-stable, and 
\item for some (and hence any, by (1)) embedding $F \rightarrow \RR,$ the $F$-rank of $G_0$ is equal to its $\RR$-rank, 
\end{enumerate}
then condition \ref{rationalparabolic} holds.
\end{lem}
\begin{proof} For the fact that $G^{\ad}$ has the form $\Res_{F/\Q} G_0$ see \cite[2.3.4(a)]{Deligne}. 
Now suppose that the conditions (1) and (2) are satisfied. 
Fix an embedding $F \rightarrow \RR,$ so that $G_{0,\RR}$ is a factor of $G^{\ad}_{\RR},$ 
and corresponds to a connected component $\Delta(G)_0$ of the Dynkin diagram of $G.$ 
Let $\tilde \varSigma_0 = \tilde \varSigma \cap \Delta(G)_0,$ and let $\mu_{\tilde \varSigma_0}$ be 
the corresponding $G^{\ad}_{\RR,0}$-valued cocharacter. 

By \cite[\S 3.2]{AMRT}, the conjugacy class of the parabolic subgroup of $G_{0,\RR}$ corresponding to $\mu_{\tilde \varSigma_0}$ 
contains a parabolic defined over $\RR,$ and hence the conjugacy class 
$[\mu_{\tilde \varSigma_0}]$ contains a cocharacter defined over $\RR,$ as in the proof of 
\ref{lem:equivrational}. By (2), after conjugation by an element of $G_0(\RR),$ this cocharacter, which we again denote by 
$\mu_{\tilde \varSigma_0},$ factors through an $\RR$-split torus which is defined over $F,$ and $F$-split. 
Thus we have a map of $F$-groups $\mu_{\tilde \varSigma_0}: \GG_m \rightarrow G_0.$

By definition of the restriction of scalars, $\mu_{\tilde \varSigma_0}$ induces 
a map $\mu_{\tilde \varSigma}': \GG_m \rightarrow G^{\ad}$ over $\Q.$ 
To see that $\mu_{\tilde \varSigma}' \in [\mu_{\tilde \varSigma}],$ note that, by (1), $[\mu_{\tilde \varSigma}]$ is $\Gal(\bar \Q/\Q)$-stable. 
Thus it suffices to check that the projections of $[\mu_{\tilde \varSigma}]$ and $[\mu_{\tilde \varSigma}']$ onto $G_0$ are equal. 
But these are both equal to $[\mu_{\tilde \varSigma_0}].$ 

%
%
 \end{proof}

\begin{para}\label{para:rationalcocharacter} Let us explain how to apply Lemma \ref{lem:rationalcocharacter} in examples. 
First, the condition (1) implies that 
$X$ has the form $X_0^m$ for some irreducible Hermitian symmetric domain $X_0.$ When $X_0$ is not of type $D_n^{\mathbb H},$ 
then (1) is actually equivalent to this condition. In the case of type $D_n^{\mathbb H},$ there are two conjugacy classes of 
cocharacters of $G_{0,\RR},$ which are exchanged by an outer isomorphism of $G_0,$ and give rise to isomorphic (via the outer isomorphism) 
Hermitian symmetric domains. If the outer isomorphism is induced by complex conjugation, which happens when $n$ is odd, then $X$ 
having the form  $X_0^m$ still implies (1), but if $n$ is even, then the condition is stronger. 

The condition (2) can also be made more explicit in many cases: If $G$ is of type $B$ or $C$ or $D_n$ with $n$ even, then $G^{\ad}_{\RR}$ is split,  
and (2) means that $G_0$ is a split group, or (in this context) that $G$ is quasi-split. If $G$ is of type $A_n$ then (2) means that $G_0$ 
is the adjoint group of a unitary group over $F$ - that is one associated to a Hermitian form on a vector space over a quadratic CM extension of $F$ 
- and not just an inner form of such a group. In particular, when $n=1$ this covers the case of Hilbert modular varieties. 

The following lemma gives a way of constructing many examples when $X$ is irreducible.
\end{para}

\begin{cor}\label{cor:existenceShim} For any irreducible Hermitian domain $X,$ and any quadratic imaginary extension $L/\Q,$ 
there exists a Shimura datum $(G,X)$ such that $G$ is an absolutely simple group which splits over $L,$ and 
\ref{rationalparabolic} holds. 
\end{cor}
\begin{proof} Let $G_{\RR}$ be the absolutely simple reductive $\RR$-group associated to $X,$ and $G_{0,\RR}$ its quasi-split inner form.  
Consider the Dynkin diagram of $G_{0,\RR}$ with its action of $\Gal(\CC/\RR) \simeq \Gal(L/\Q).$ This corresponds to a quasi-split group 
$G_0$ over $\Q,$ which splits over $L.$ Now by \cite[Thm.~1]{PrasadRapinchuk}, there exists an inner form $G$ of $G_0,$ which 
becomes isomorphic to $G_{\RR}$ over $\RR$ (as an inner form of $G_{0,\RR}$), and whose $\Q$-rank is equal to the $\RR$-rank of 
$G_{0,\RR}.$ (Note that in {\em loc.~cit.~}the number of $i$ satisfying the condition (ii) with $v = \infty$ is precisely the split rank of $G_{\RR}.$).
The corollary now follows from Lemma \ref{lem:rationalcocharacter}.
\end{proof}

\begin{para} We will now apply the theory of toroidal compactifications of Shimura varieties to our situation. We refer the reader 
to the book of Ash-Mumford-Rapoport-Tai \cite{AMRT}, and Pink's thesis \cite{Pinkthesis}; see also Madapusi-Pera's paper 
\cite[\S 2]{Keerthi} for a convenient summary. 
We suppose from now on that the condition \ref{rationalparabolic} holds. 

Let $P \in [P_{\tilde \varSigma}]$ be defined over $\Q,$ and let $U \subset P$ be its unipotent radical, 
and $Z \subset U$ the center of of $U.$ Abusing notation slightly, write $U(\Z) = U(\Q) \cap K,$ $Z(\Z) = Z(\Q) \cap K.$ 
Let $\xi_K = U(\Z)\backslash U(\RR)Z(\CC),$ and $C_K = (Z(\CC)/Z(\Z))\backslash \xi_K.$ 
Then $C_K$ is an abelian scheme and $\xi_K \rightarrow C_K$ 
is a torus bundle. Fix a fan $\Sigma,$ in $X^*(Z(\CC)/Z(\Z))_{\Q}$ so that the resulting torus embedding is smooth with a normal crossings 
divisor at the boundary. As in Lemma \ref{classificationtoric}, we obtain a smooth torus embedding over $C_K,$ $\xi_K(\Sigma) \rightarrow C_K$ 
whose boundary $D_K(\Sigma)$ is a normal crossings divisor, see \cite[\S 3]{Pinkthesis}, \cite[2.1.7]{Keerthi}. 

Note that in the constructions of {\em loc.~cit.~}one gets a smooth torus embedding over a base which is itself an abelian scheme over a Shimura variety. 
That is because those constructions start with the preimage under $P \rightarrow P/U$ of a factor of the reductive group $P/U.$ However, for 
our specific choice of parabolic $P_{\tilde \varSigma}$ the relevant factor of $P/U$ is a torus, and the Shimura variety is $0$-dimensional. 
In the notation of \cite[\S 4.7]{Pinkthesis}, our parabolic $P$ is what Pink denotes by $Q,$ and the preimage mentioned above 
is denoted by $P_1.$ These groups are denoted by $P$ and $Q_P$ in \cite[\S 2]{Keerthi}. 
\end{para}

\begin{para} By an {\em analytic neighborhood of $D_K(\Sigma)$} we mean an analytic open subset 
$V_K \subset \xi_K(\Sigma)(\CC),$ containing $D_K(\Sigma).$ We write $V_K^* = V_K\backslash D_K(\Sigma).$ 

Write $\Sh_K = \Sh_K(G,X).$ Fix a connected component $X^+ \subset X,$ and let $\Sh_K^+ \subset \Sh_K$ be the corresponding connected 
component of $\Sh_K^+.$ Then $\Sh^+_K = \Gamma \backslash X^+,$ where $\Gamma = K \cap G(\Q)_+,$ and $G(\Q)_+ \subset G(\Q)$ 
is the subgroup mapping to the connected component of the identity in $G^{\ad}(\RR)$ \cite[2.1.2]{Deligne}. In particular, for any choice of base point, 
this allows us to identify the fundamental group of $\Sh_K^+$ with $\Gamma.$

For any smooth $\CC$-scheme $Y,$ by an {\em ncd compactification} of $Y$ we mean a dense embedding $Y \subset \bar Y$ 
into a proper smooth $\CC$-scheme $\bar Y,$ such that $\bar Y\backslash Y$ is a normal crossings divisor. 
We collect the results we need in the following proposition. 
\end{para}

\begin{prop}\label{prop:toroidalsummary} There exists an ncd compactification $\Sh_K(\Sigma)^+$ of $\Sh_K$ and 
an analytic neighborhood $V_K$ of $D_K(\Sigma)$ such that 
\begin{enumerate}
\item The inclusion $V_K^* \rightarrow \xi_K(\CC)$ induces an isomorphism of topological fundamental groups.
\item There is an \'etale map of complex analytic spaces $$ \pi_K: V_K \rightarrow \Sh_K(\Sigma)^+$$ 
such that $\pi_K^{-1}(\Sh_K(\Sigma)^+\backslash\Sh_K^+) = D_K(\Sigma),$ 
and $\pi_K: V_K^* \rightarrow \Sh_K^+$ induces the natural map $U(\Z) \hookrightarrow \Gamma$ on fundamental groups.
\end{enumerate} 
\end{prop}
\begin{proof} By \cite[4.11]{Pinkthesis}, \cite[2.1.6]{Keerthi} there is an open immersion of complex analytic spaces 
$X^+ \rightarrow U(\RR)Z(\CC).$ Now set $V_K^* = U(\Z)\backslash X^+,$ then there is an induced open immersion $V_K^* \rightarrow \xi_K(\CC)$ (see \cite[6.10]{Pinkthesis}, 
\cite[2.1.13]{Keerthi}) which induces an isomorphism on fundamental groups, by construction. 
By \cite[6.13]{Pinkthesis}, \cite[2.1.22]{Keerthi}, the latter open immersion extends to an analytic neighborhood $V_K$ of $D_K(\Sigma).$ 
The existence of  $\Sh_K(\Sigma)^+$ and the map $\pi_K$ with the properties in (2) is the main result of \cite[\S 6]{Pinkthesis}, see also 
\cite[2.1.26]{Keerthi}. We remark that the ncd compactification $\Sh_K(\Sigma)^+,$ depends on more choices than just $\Sigma,$ 
but as these will play no role for us, we omit them from the notation.
\end{proof}

\begin{para} We continue to assume that condition \ref{rationalparabolic} holds. Suppose that $G$ admits a reductive model $G_{\Z_p}$ over $\Z_p.$ 
As the scheme of parabolic subgroups is projective 
\cite[XXVI, Cor.~3.5]{SGA3}, $P$ extends to a parabolic subgroup $P_{\Z_p} \subset G_{\Z_p}.$  Denote by $U_{\Z_p} \subset P_{\Z_p}$ the unipotent 
radical. We will sometimes write $G,$ $P$ and $U$ for $G_{\Z_p},$ $ P_{\Z_p}$ and $U_{\Z_p}$ if this causes no confusion. 
\end{para}

\begin{lem}\label{lem:Heisenbergsetup} The group $U_{\Z_p}(\Z_p)$ is a central extension of finitely generated, free abelian pro-$p$ groups. 
The map $U(\Z_p) \rightarrow U(\F_p)$ is surjective, and $U(\F_p)$ is a reduction of $U(\Z_p)$ mod $p.$
\end{lem}
\begin{proof} As in the proof of Lemma \ref{lem:equivrational}, one sees that $P_{\Z_p} \subset G_{\Z_p}$ corresponds to a cocharacter 
$\mu_{\Z_p}$ in $[\mu_{\tilde \varSigma}],$ defined over $\Z_p.$ As $U_{\Z_p}$ is unipotent, it is an iterated extension of additive groups. 
It then follows by Lemma \ref{lem:strucunipot} that $U_{\Z_p}$ is a central extension of 
additive groups, so that $U(\Z_p)$ is a central extension of finitely generated, free, abelian pro-$p$ groups. 
As $U_{\Z_p}$ is smooth, $U(\Z_p) \rightarrow U(\F_p)$ is surjective.

To check that $U(\F_p)$ is a reduction of $U(\Z_p)$ mod $p,$ we have to check that $U(\F_p)$ is a Heisenberg group. 
This condition is vacuous if $U$ is abelian. If $U$ is not abelian, then 
by Lemma \ref{lem:strucunipot}, $\mu_{\Z_p}$ has weight $2$ on $Z(U_{\Z_p})$ and weight $1$ on $U_{\Z_p}/Z(U_{\Z_P}).$ 
Hence $\F_p^{\times}$ acting via $\mu_{\Z_p}$ has weight $2$ on $Z(U_{\Z_p})(\F_p)$ and weight $1$ on  $U_{\Z_p}/Z(U_{\Z_p})(\F_p).$ 
Thus $U(\F_p)$ is a Heisenberg group by Lemma \ref{lem:Heisenbergcondns}.
\end{proof}

\begin{para} Now suppose that $K = K_p K^p,$ with $K_p = G(\Z_p),$ and $K^p \subset G(\AA_f^p)$ compact open.
Let $K^1_p = \ker(G(\Z_p) \rightarrow G(\F_p)),$ and $K^1 = K^1_pK^p.$ As above, let $\Gamma = K \cap G(\Q)_+,$ 
and set $\Gamma_1 = K^1 \cap G(\Q)_+.$ A covering, $\Gamma_1\backslash X^+ \rightarrow \Gamma\backslash X^+,$ 
with $\Gamma, \Gamma_1$ of the above form is called a {\em principal $p$-covering}.
\end{para}

\begin{thm}\label{thm:edlsym} Assume that \ref{rationalparabolic} holds, and that $G$ extends to a reductive group scheme over 
$\Z_p.$ 
If $\varSigma \neq c(\varSigma)$ then we assume that $p > \dim X + \frac{1}{2} (\dim Z + 1),$ and that $p$ is a prime of unramified 
good reduction for $\xi_K.$ 
Then 
$$ \ed(\Gamma_1 \backslash X^+ \rightarrow \Gamma \backslash X^+; p) = \dim X. $$ 
\end{thm}  
\begin{proof} Let $V_K$ be an analytic neighborhood as in Prop.~\ref{prop:toroidalsummary}, and fix a base point $\bar s \in V_K.$ 
By Proposition \ref{prop:toroidalsummary}(1),  we may identify $\pi_1(V^*_K, \bar s)$ with $\pi_1(\xi_K(\CC), \bar s) \simeq U(\Z).$ 
By the strong approximation for unipotent groups, the pro-$p$ completion of $U(\Z)$ is $U(\Z_p).$ 

The pullback of $\Gamma_1 \backslash X^+ \rightarrow \Gamma \backslash X^+$ by the map $\pi_K$ in Proposition \ref{prop:toroidalsummary}(2) 
 is a covering, one of whose components $V^*_{K^1} \rightarrow V_K^*,$ corresponds to the kernel of the composite 
$$ U(\Z) \rightarrow \Gamma \rightarrow K \rightarrow G(\Z_p) \rightarrow G(\F_p). $$
That is, it is the kernel of $U(\Z) \rightarrow U(\F_p).$ 
Using the strong approximation for unipotent groups, as above, it follows that $V^*_{K^1} \rightarrow V^*_K$ is a covering with group $U(\F_p),$ 
and is the restriction of a covering of {\em schemes} $\xi_{K^1} \rightarrow \xi_K$ with group $U(\F_p).$ 
Let $\xi_{K^1}(\Sigma)$ be the normalization of $\xi_K(\Sigma)$ in $\xi_{K^1},$ and set $V_{K^1} = \xi_{K^1}(\Sigma)^{\an}|_{V_K}.$

Next suppose $B \subset \Sh^+_K = \Gamma \backslash X^+$ is a Zariski closed subset. Let $\bar B$ be its Zariski closure in 
$\Sh^+_K(\Sigma).$ Then $\pi_K^{-1}(\bar B)$ is Zariski closed in $V_K.$ It follows that  
$$  \ed(\Gamma_1 \backslash X^+ \rightarrow \Gamma \backslash X^+; p) \geq \ed(V_{K^1} \rightarrow V_K; p).$$
Here we are using the notion of $p$-essential dimension for analytic spaces introduced in \ref{analyticedp}.

As $U(\F_p)$ is a mod $p$ reduction of $U(\Z_p)$ by Lemma \ref{lem:Heisenbergsetup}, 
the theorem follows from Corollary \ref{cor:analyticedp} if $ c(\varSigma) \neq \varSigma$ and from Corollary \ref{cor:analyticedp2} 
if $c(\varSigma) = \varSigma.$
\end{proof}

\begin{cor}\label{cor:edlhodge} With the assumptions of Theorem \ref{thm:edlsym},  suppose that $(G,X)$ is of Hodge type. 
Then $p$ is a prime of unramified good reduction for $\xi_K,$ and the conclusion of the theorem holds without this assumption.
\end{cor}
\begin{proof} This is a consequence of the main result of Madapusi-Pera \cite{Keerthi}. 
\end{proof}

\begin{cor}\label{cor:edlsym2} Let $X$ be an irreducible symmetric domain, and 
let $L/\Q$ be a quadratic extension. 
Then there exists a Shimura datum $(G,X)$ with $G$ an absolutely simple group which splits over $L,$ such that 
for any principal $p$-covering $\Gamma_1\backslash X^+ \rightarrow \Gamma \backslash X^+,$ we have 
$$ \ed(\Gamma_1\backslash X^+ \rightarrow \Gamma \backslash X^+; p ) = \dim X,$$ 
provided $p$ satisfies the following conditions if $X$ is not a tube domain:
\begin{itemize} 
\item If $X$ is of classical type, then $p > \frac{3}{2}\dim X.$ 
\item If $X$ is of type $E_6$, then $p$ is sufficiently large.
\end{itemize}
\end{cor}
\begin{proof} Apply Corollary \ref{cor:existenceShim}, to obtain a Shimura datum $(G,X)$ such that $G$ is an absolutely simple group which splits over $L,$ and such that \ref{rationalparabolic} holds. Our definition of principal $p$-coverings already assumes that $G$ admits a reductive model over $\Z_p,$ so we assume this from now on.

If $X$ is of tube type (that is $c(\varSigma) = \varSigma$), or $X$ is of type $E_6,$ then the corollary follows from Theorem \ref{thm:edlsym}. 

Suppose $X$ is not a tube domain and is of classical type. Then $(G,X)$ is of abelian type \cite[2.3.10]{Deligne}. Recall that this means 
(since  $G$ is adjoint), that there is a morphism of Shimura data $(G',X') \rightarrow (G,X),$ induced by a 
central isogeny $G' \rightarrow G,$ with $(G',X')$ of Hodge type. As in \cite[3.4.13]{Kisin}, we can assume that $G'$ again has a reductive 
model over $\Z_p.$ 

Now we apply \ref{cor:edlhodge} to a principal $p$-covering 
$\Gamma_1'\backslash X^+ \rightarrow \Gamma'\backslash X^+,$ 
coming from the group $G'.$ 
We obtain the $p$-incompressibility of this covering for $p > \dim X + \frac{1}{2}(\dim Z+1).$ 
As 
$$\dim X = \frac{1}{2} (\dim U + \dim Z) = \frac{1}{2}\dim (U/Z) + \dim Z,$$ we have 
$$ \dim X + \frac{1}{2}(\dim Z+1) = \dim X +  \frac{1}{2}(\dim X - \frac{1}{2}(\dim U/Z)  + 1) \leq \frac{3}{2} \dim X. $$
Now a calculation as in \cite[4.3.12]{FKW} shows that the kernel and cokernel of 
$\Gamma_1'\backslash \Gamma' \rightarrow \Gamma_1\backslash \Gamma$ are finite groups of order prime to $p.$ 
The result now follows by \cite[2.2.7]{FKW}.
\end{proof}

\bibliographystyle{apa}

\end{document}